\documentclass[12pt,reqno]{amsart}
\usepackage{fullpage}
\usepackage{times}
\usepackage[T1]{fontenc}
\usepackage[utf8]{inputenc}
\usepackage{amsmath,amsfonts,amssymb,amsxtra,setspace,xspace,graphicx,lmodern,fourier,mathrsfs}
\usepackage[colorlinks=true]{hyperref}
\hypersetup{urlcolor=blue, citecolor=red, linkcolor=blue}
\usepackage[usenames,dvipsnames]{color}
\usepackage[misc]{ifsym}
\usepackage{cleveref}

\newcommand{\R}{{\mathbb R}}

\newcommand{\be}[1]{\begin{equation*}\label{#1}}

\newcommand{\ee}{\end{equation*}}
\renewcommand{\(}{\left(}
\renewcommand{\)}{\right)}
\newcommand{\ird}[1]{\int_{\R^N}{#1}\,dx}

\newcommand{\ire}[1]{\int_{\R^N}{#1}\,dy}

\newtheorem{thm}{Theorem}[section]
\newtheorem{cor}[thm]{Corollary}
\newtheorem{lem}[thm]{Lemma}
\newtheorem{rmk}[thm]{Remark}
\newtheorem{prop}[thm]{Proposition}
{ \theoremstyle{remark} }

\begin{document}

\title[Large time asymptotic behaviours of two types of fast diffusion equations]{Large time asymptotic behaviors of two types of fast diffusion equations}

\author{Chuqi Cao, Xingyu Li }
\thanks{C. Cao: Yau Mathematical Science Center and  Beijing Institute of Mathematical Sciences and Applications, Tsinghua University, Beijing, China. \\
Email address: chuqicao@gmail.com \\[1pt]}
\thanks{X. Li: CEREMADE (CNRS UMR n.7534), PSL research university, Universit\'e Paris-Dauphine, Place de Lattre de Tassigny, 75775 Paris 16, France.\\
Email address: li@ceremade.dauphine.fr}%
\maketitle\thispagestyle{empty}

\begin{abstract} We consider two types of nonlinear fast diffusion equations in $\mathbb{R}^N$: \\

(1) External drift type equation with general external potential.  It is a natural extension of the harmonic potential case, which has been studied in many papers. In this paper we can prove the large time asymptotic behavior to the stationary state by using entropy methods. \\

(2) Mean-field type equation with the convolution term. The stationary solution is the minimizer of the free energy functional, which has direct relation with reverse Hardy-Littlewood-Sobolev inequalities. In this paper, we prove that for some special cases, it also exists large time asymptotic behavior to the stationary state.
\end{abstract}

\section{Introduction}

Asymptotic rates of convergence for the solutions of diffusion equations have been studied in many papers, such as \emph{linear diffusion equation}(Keller-Segel type, see \cite{hoffmann2017keller}), \emph{porous medium type equation}(see \cite{carrillo2000asymptotic}).
In this paper, we mainly study the asymptotic behavior of two different types of non-linear \emph{fast diffusion equations} in $\mathbb R^N$, and we denote 
\[
V_{\lambda}(x):=\frac{1}{\lambda}|x|^{\lambda} (\lambda>0), \frac{N}{N+\lambda}<q<1\]
and
\[
q_{\#}:=\left\{
\begin{aligned}
  & \frac{N-2-\lambda}{N-2} , &\quad  N\ge 3 \\
  & 0, &\quad N=1,2 
\end{aligned}
\right.:
\]
\\

\textbf{1. Fast diffusion equation with external drift}
\begin{equation}
\label{fd1}
n_t=\Delta(n^q)+\nabla(n\nabla V_{\lambda}),\quad n(0,.)=n_0>0.
\end{equation}
Obviously, this equation is mass conserved. Set $\ird{n(x)}=m>0$.  Notice that the stationary solution $N_h$ of \eqref{fd1} is of the form
\begin{equation*}
N_h=\left(\frac{1-q}{q}(h+V_{\lambda})\right)^{\frac{1}{q-1}}.
\end{equation*}
here $h(m)>0$ is uniquely determined by the equation
\begin{equation}
\label{mass}
\ird{N_h}=m.
\end{equation}
From now on, we always suppose that $n$ satisfies the following assumption if not specially mentioned.\\
{\label{as1}
\textbf{(H1)} There exist constants $0<h_1<h_2$, such that
\[
N_{h_2}(x)\le n_0(x)\le N_{h_1}(x), \quad \forall x \in \R^d
\]}
According to the maximum principle, $N_{h_2}\le n(t,.)\le N_{h_1}$ for any $t>0$, see page 8-10 from \cite{blanchet2009asymptotics} for more details. We will also denote $h_*$ which is uniquely determined by
\[
\int_{\R^N} N_{h_*}  dx= \int_{\R^N} n_0 dx
\]

\textbf{2. Mean-field type fast diffusion equation} \\

We consider the equation 
\begin{equation}
\label{fd2}
\rho_t=\Delta(\rho^q)+\nabla(\rho\nabla V_\lambda*\rho), \quad\rho(0,.)=\rho_0>0.
\end{equation}
For this equation, we always suppose that $\frac{N}{N+\lambda}<q<1$. Notice that the stationary solution $\rho_{\infty}$ satisfies
\[
\frac{q}{1-q}\rho_{\infty}^{q-1}=V_\lambda*\rho_{\infty}+C
\]
for some constant $C>0$. In this paper, we focus on the normalized mass case 
\[
\ird{\rho(t, x)}=\ird{\rho_{\infty}}=1
\]
we mention here equation \eqref{fd2} with other mass can be treated similarly. Notice that for $i=1,...N$,
\[
\frac{d}{dt}\ird{x_i\rho}=\ird{\ire{\rho(x)|x-y|^{\lambda-2}(x_i-y_i)\rho(y)}}=0
\]
from the result of \cite{dolbeault2018reverse}, we can assume that $\rho_{\infty}$ is radially symmetric and non-increasing, so 
\[
\ird{x_i\rho_0}=\ird{x_i\rho_\infty}=0
\]
moreover, we also assume that\\
{\label{as2}
\textbf{(H2)} There exist two stationary solutions (possibly with different mass) $\rho_1, \rho_2$ such that
\[
\rho_{2}(x)\le \rho_0(x)\le \rho_{1}(x), \quad \forall x \in \R^d.
\]}
When $\lambda=2$, we know from \cite{dolbeault2018reverse} that after modulo translations,
 $\rho_{\infty}$ has the form
\begin{equation}
\label{case2}
\frac{q}{1-q}\rho_{\infty}^{q-1}=\frac{1}{2}|x|^2+C
\end{equation}
by the theory of Beta function, $C$ satisfies
\[
C^{\frac{1}{1-q}-\frac{N}{2}}=(2\pi)^\frac{N}{2}\left(\frac{1-q}{q}\right)^\frac{1}{q-1}\frac{\Gamma\left(\frac{1}{1-q}-\frac{N}{2}\right)}{\Gamma\left(\frac{1}{1-q}\right)}
\]
and we will show the large time asymptotic to the stationary solution  $\rho_{\infty}$. 

For $\lambda \neq 2$, the form of $\rho_{\infty}$ is much more complicated so it's hard to analyze it.  From the result of \cite{dolbeault2018reverse}, when $q \in (\frac {N } {N+\lambda}, 1 )$, the stationary solution $\rho_\infty$ satisfies
\[
\rho_\infty =  \rho_* +M_* \delta_0
\]
for some $M_* \in [0, 1)$ and $\rho_* \in L^1_+ \cap L^q(\R^N)$ is radically symmetric and supported on $\R^N$, and $\delta_0$ denotes the Dirac measure at the point 0. When $q(\frac {N } {N+\lambda}, 1 )$ near $\frac {N } {N+\lambda}$,  in a recent preprint \cite{carrillo2020fast} the authors proved that $M_*>0$ with $\lambda =4, N \ge 6$, but when $q>\frac{2N}{2N+\lambda}$, it is proved in \cite{dolbeault2018reverse} that the Dirac mass does not appear in $\rho_{\infty}$. In this paper, we mainly consider the case $\lambda>2$ and suppose that $\frac{2N}{2N+\lambda}<q<1$ without more explication. We will try to analyze the exact result about $\rho_{\infty}$ and also show the similar asymptotic behavior if $q$ is close enough to 1.

\subsection{Main tools and results.}
For equation \eqref{fd1}, we consider the \emph{free energy}
\begin{equation}\label{free1}
\mathcal F[n]:=-\frac{1}{1-q}\ird{n^q}+\ird{V_{\lambda} n}
\end{equation}
the \emph{relative entropy}
\[
\mathcal{F}[n|N_h]:= \mathcal F[n]- \mathcal F[N_h]
\]
and the \emph{relative Fisher information} with respect to $N_{h_*}$ defined as
\begin{equation}\label{fisher1}
\mathcal I[n]:=\ird{n\left|\nabla\left(\frac{q}{q-1}n^{q-1}+V_{\lambda}\right)\right|^2}=\frac{q^2}{(1-q)^2}\ird{n|\nabla(n^{q-1}-N_{h_*}^{q-1})|^2}
\end{equation}
we will prove in Section~\ref{Sec2:fastdiffu1} that under the constraints $n(x)>0,\ird{n(x)}=m$, $\mathcal F$ is bounded from below, and $N_{h_*}$ is the unique minimizer , which is $\mathcal F[n]-\mathcal F[N_{h_*}]\ge 0$. If $n$ solves
\eqref{fd1}, we obtain that 
\[
\frac{d}{dt}\mathcal F[n]=-\mathcal I[n]
\]
see Proposition 2 of \cite{blanchet2009asymptotics} for more details. \\
For equation~\eqref{fd2}, the \emph{free energy} becomes
\begin{equation}
\label{free5}
\mathbb{F}[\rho]=-\frac{1}{1-q}\ird{\rho^q}+\frac{1}{2\lambda}
 I_{\lambda}[\rho]
\end{equation}
the \emph{relative entropy}
\[
\mathbb{F}[\rho|\rho_\infty]:= \mathbb F[\rho]- \mathbb F[\rho_\infty]
\]
where
\[
 I_{\lambda}[\rho]=\ire{\ird{|x-y|^{\lambda}\rho(x)\rho(y)}}
\]

notice that because $\frac{N}{N+\lambda}<q<1$, the free energy of $\mathcal F[\rho_{\infty}]$ is finite. Moreover, we define the \emph{relative Fisher information}
\begin{equation}
\label{free6}
\mathbb{I}[\rho]:=\ird{\rho\left|\frac{q}{1-q}\nabla\rho^{q-1}-\nabla V_\lambda*\rho\right|^2}
\end{equation}
our analysis is based on the Theorem below.
\begin{thm}(\cite{dolbeault2018reverse}, Reverse Hardy-Littlewood-Sobolev inequality)
Let $N \ge 1 , \lambda>0, q \in(0, 1)$, define
\[
\alpha :=\frac {2N - q(2N+\lambda) } {N(1-q)} 
\]
then the inequality
\begin{equation} 
\label{reversehls1}
\ire {\ird { |x-y|^\lambda   \rho(x) \rho(y) }}\ge C_{N, \lambda, q}  \left(\ird{\rho(x)} \right)^\alpha \left(\ird{\rho(x)^q }\right)^{\frac {2-\alpha} {q}}
\end{equation}
holds for any nonnegative function $\rho \in L^1 \cap L^q(\R^N)$ and for some positive constant $C_{N, \lambda, q}$ if and only if $q>\frac {N} {N+\lambda}$, or equivalently $\alpha<1$.
Moreover, the optimal function of the inequality is the global minimizer of  $\mathbb F[\rho]$. The optimal function is radially symmetric, non-increasing and supported
on $\mathbb R^N$, and it is unique up to translation.
\end{thm}
Because $q>\frac{N}{N+\lambda}$, from the results of \cite{dolbeault2018reverse}, the free energy $\mathbb F[\rho]$ is bounded from below under the \emph{mass} constraint, and energy minimizers of $\mathcal F$ are the optimal functions of \eqref{reversehls1}. Moreover, $\rho_{\infty}$ is the unique minimizer up to translation, and
\[
\mathbb{I}[\rho]=-\frac{d}{dt}\mathbb{F}[\rho].
\]
Our goal is to show that
\begin{thm}
\label{thm1.1}
Suppose that the solution $n$ of the equation \eqref{fd1} with initial data $n_0$ satisfying (\textbf H\ref{as1}), $q \in\left(\frac {N} {N+\lambda}, 1 \right)$ and $\mathcal F[n_0]<\infty$. Then for $\lambda \ge 2$, there exist constants $\kappa, \mu>0$, such that for any $t>0$,
\[
\ird{N_{h_*}^{q-2}(n(t,.)-N_{h_*})^2}\le \kappa e^{-\mu t}.
\]
For $\lambda \in (0, 2)$, if we assume in addition $q\in\left(\frac {N+2} {N+2+\lambda},1\right)$,
then there exist constants $\kappa, \mu>0$, such that for any $t>0$,
\[
\ird{N_{h_*}^{q-2}(n(t,.)-N_{h_*})^2}\le \kappa(1+t)^{-\mu }.
\]
\end{thm}
For equation \eqref{fd2} we have similar results when $\lambda =2$.
\begin{thm}
\label{thm1.3}
For $\lambda=2$, suppose that the solution $\rho$ of the equation \eqref{fd2} with initial data $\rho_0$ that satisfies (\textbf H2) and $\mathbb F[\rho_0]<\infty$. Then for all $q\in\left(\frac{N}{N+2},1\right)$, there exist constants $\tau, \gamma>0$, such that for any $t>0$,
\begin{equation*}
\ird{\rho_\infty^{q-2}(\rho(t,.)-\rho_\infty)^2}\le \tau e^{-2\gamma t}.
\end{equation*}
\end{thm}
For general $\lambda>2$, we have the similar result when $q$ is near 1.
\begin{thm}
\label{thm1.4}
Suppose that $\lambda>2$, and the solution $\rho$ of the equation \eqref{fd2} with initial data $\rho_0$ that satisfies (\textbf H2) and $\mathbb F[\rho_0]<\infty$. Then
there exists a constant $q_{N,\lambda}\in (\frac{2N}{2N+\lambda}, 1)$, such that for any $q\in (q_{N,\lambda}, 1)$ , there exist constants $\tau, \gamma>0$ such that for any $t>0$,
\[
\mathbb{F}[\rho|\rho_\infty]\le\tau e^{-2\gamma t}\mathbb{F}[\rho_0|\rho_\infty]
\]
in particular if $\lambda \in (2,4]$, we have
\begin{equation*}
\ird{\rho_\infty^{q-2}(\rho(t,.)-\rho_\infty)^2}\le \tau e^{-2\gamma t}.
\end{equation*}
\end{thm}

\begin{rmk}
We remark here that for equation \eqref{fd1} if we assume further that \\
(\textbf H1')There exists a constant $h_*\in[h_1,h_2]$, such that
\begin{equation*}
p(x):=n_0(x)-N_{h_*}(x)\in L^1(\mathbb{R}^d).
\end{equation*}
Then we can extend the Theorem \ref{thm1.1} to the case $q\in (0,1)$ when $N=1$ or 2, and $q \in (0, 1)\backslash\{\frac{N-2-\lambda}{N-2}\}$ when $N\ge 3$ with the similar method, including the pseudo-Barenblatt solutions case that $q \in (0, \frac{N-2-\lambda}{N-2})$ when $N\ge 3$. The readers are invited to check \cite{vazquez2006smoothing, blanchet2009asymptotics} for more information.
\end{rmk}

\subsection{Background}
The nonlinear fast diffusion equations have caught many attentions, see \cite{vazquez2006smoothing} for a more precise introduction. And studying the asymptotic rates of convergence to the stationary states is an important theme. For the equation \eqref{fd1}, the harmonic potential case that $\lambda=2$ has been studied by \cite{dolbeault2018reverse}. The result is based on the spectral method of the linearized equation, and the optimal rate of the convergence can be directly deduced by the spectral gap, which is the optimal constant of Hardy-Poincar\'e inequality. See \cite{blanchet2007hardy} for more details. The mean-field equation \eqref{fd2} is more complicated, since mean field potential $W_\lambda = V_\lambda*\rho$ depends on the regular part $\rho$, and for more general $\lambda$, there is no explicit form of $\rho_{\infty}$ for the estimate. This equation behaves different with different choice of $\lambda$ and $q$. Recall the functional energy 
\[
\mathbb{F}[\rho]=-\frac{1}{1-q}\ird{\rho^q}+\ire{\ird{V_{\lambda}(x-y)\rho(x)\rho(y)}}
\]
if we take the mass-preserving dilations
\[
\rho_\lambda(x) = \beta^N \rho(\beta x)
\]
so
\[
\mathbb{F}[\rho]=-\frac{\beta^{N(q-1) }}{1-q}\ird{\rho^q}+\beta^{-\lambda}\ire{\ird{V_{\lambda}(x-y) \rho(x)\rho(y)}}
\]
and one observes different types of behavior depending on the relation between the parameters $N, q$ and $\lambda$. The energy functional is homogeneous if attraction and repulsion are in balance, so that the two terms of the energy scale with the same power, that is, $q=q_*$ with
\[
q_* = 1- \frac \lambda N
\]
this motivates the definition of three different regimes: the \emph{diffusion-dominated regime} $q > q_*$, the \emph{fair-competition regime} $q = q_*$, and the \emph{attraction-dominated regime} $0<q<q_*$. We refer to \cite{hoffmann2017keller} for a complete summary of existing works. The fair-competition case has been studied in several papers such as \cite{calvez2017equilibria, calvez2019uniqueness}. 

In our paper we consider the case $q>\frac {N} {N+\lambda} > 1-\frac {\lambda} N =q_*$, which correspond to the diffusion-dominated regime. In the diffusion dominated regime, several results have been done for the case $-N< \lambda <0$ and $q>1$, the logarithmic case $\lambda = 0, q>1$ in two dimensions \cite{carrillo2015ground}, and the Newtonian case $\lambda = 2-N$ in \cite{blanchet2009critical,kim2012the}. For our case $\lambda >0, q>\frac {N} {N+\lambda}$, the existence and uniqueness of $\rho_\infty$ has been given in \cite{dolbeault2018reverse}, we remark here that  the asymptotic behavior of the case $\lambda >0, q>\frac {N} {N+\lambda}$ is to our knowledge new.

 As we introduced above, the free energy and the related Fisher information are key tools, but we need the propositions about the bound and the minimizers of the free energy.  For $q<1$, the reverse Hardy-Littlewood-Sobolev inequalities in \cite{dolbeault2018reverse, dou2015reverse, ngo2017sharp} provide the sufficient conditions. Intersted readers can check \cite{dolbeault2018reverse} for more information. 
 
 \subsection{Sketch of the proof}
In this paper, the proof for the paper is to use linearization around equilibrium plus nonlinear stability, roughly speaking, for a mass conserving equation
\[
\partial_t f= L f
\]
with equilibrium $f_\infty$, we linearize the operator $L=L_1+L_2$ (where $L_1$ is linear) around equilibrium $g=f-f_\infty$, which is
\[
\partial_t g =L_1 g+L_2g
\]
then we first prove the convergence for the linearized eqution
\[
\partial_t g=L_1g
\]
then we use the nonlinear stability to prove that the convergence results still holds for the nonlinear equation when the initial data $f_0$ is close to equilibrium
\[
\Vert f_0 - f_\infty\Vert_{X} \le \epsilon
\]
for some space $X$ and some $\epsilon>0$ small, then we use a (weak) global convergence to prove that, for any $f_0 \in X$, we can find a time $t_0>0$  such that
\[
\Vert f(t) - f_\infty\Vert_{X} \le \epsilon, \quad \forall  t \ge t_0
\]
the global convergence can be very weak, in this paper we use
\[
\lim_{t\to \infty} \Vert f(t)-f_\infty \Vert_X=0
\]
which can be proved by the entropy method, gathering the two things we complete the proof of the asymptotic behavior for the full nonlinear problem . Such method can improve convergence rate and  get better rate of convergence at large time. It's largely used in the asymptotic behavior of many nonlinear equations, see \cite{carrapatoso2017landau,herau2020regularization} for its use in Boltzmann and Landau equation for example.
\subsection{Plan of the paper.} Sections~\ref{Sec2:fastdiffu1}-\ref{Sec4:fastdiffu3} are devoted to the fast diffusion with external drift. In Section~\ref{Sec2:fastdiffu1}, we give the results about the free energy and the comparison principle. In Section~\ref{Sec3:fastdiffu2}, we prove the result about the convergence without rate and the convergence with rate in Section~\ref{Sec4:fastdiffu3}. Sections~\ref{Sec5:nonlidiffu1} -\ref{Sec7large} are about the mean-field equation with convolution term. Section~\ref{Sec5:nonlidiffu1} is about some basic properties about the linearized equation. The case that $\lambda=2$ is simple, we prove the result about large time asymptotic behavior in Section~\ref{Sec6:nonlidiffu2}, and in Section~\ref{Sec7large}, we deal with more general $\lambda>2$.
\subsection{Notations}
We will denote
\[h_k(x):=\frac{x^{k-1}-1}{k-1}\]
\begin{equation}
\label{weight1}
\mathcal M(x):=\left\{
\begin{aligned}
  & \frac {1} {1+|x|^{\lambda-2}} , &\quad \lambda \ne 2 \\
  & 1, &\quad \lambda =2
\end{aligned}
\right.
\end{equation}
and $S _N$ denotes the area of unit N-dimension sphere.

\bigskip{\bf Acknowledgments.} The first author has been supported by grants from Beijing Institute of Mathematical Sciences and Applications and Yau Mathematical Science Center, Tsinghua University, and the second author has been supported by grants from Ceremade, Universit\'e Paris Dauphine. The authors thank J.Dolbeault for introducing the problem and useful comments.\\

\section{External drift type equation: some preparations}
\label{Sec2:fastdiffu1}

\subsection{The free energy and its minimizer}
We first prove the basic propositions about the free energy $\mathcal F[n]$ defined in $\eqref{free1}$, and the existence of the minimizers of $\mathcal F[n]$ under the condition $\ird{n(x)}=m$. 
\begin{prop}
The free energy $\mathcal F[n]$ satisfies\\

(1)For $h\ge 0$, the free energy $\mathcal F[N_h]$ is increasing by $h$, which means that it is decreasing by the mass $m$.\\

(2)For any $n>0, \ird{n(x)}=m$, we have $\mathcal F[n]\ge \mathcal F[N_h]$, and equality fits if and only if $n=N_h$. Here $h$ is decided by the equation \eqref{mass}. In particular, $\mathcal F[n]$ is bounded from below.
\end{prop}
\begin{proof}
(1)Remind that
\begin{equation*}
N_h=\left(\frac{1-q}{q}\right)^\frac{1}{q-1}\cdot(h+V_{\lambda})^\frac{1}{q-1}
\end{equation*}
so
\begin{equation*}
\begin{aligned}
\mathcal F[N_h]&=\frac{1}{q-1}\ird{N_h^q}+\ird{V_{\lambda}N_h}\\
&=\left(\frac{1-q}{q}\right)^\frac{1}{q-1}\cdot\left(-\frac{1}{q}\ird{(h+V_{\lambda})^\frac{q}{q-1}}+\ird{V_{\lambda}(h+V_{\lambda})^\frac{1}{q-1}}\right)
\end{aligned}
\end{equation*}
which means that
\begin{equation*}
\begin{aligned}
\frac{d}{d h}\mathcal F[N_h]
&=\left(\frac{1-q}{q}\right)^\frac{1}{q-1}\cdot\left(-\frac{1}{q-1}\ird{(h+V_{\lambda})^\frac{1}{q-1}}+\frac{1}{q-1}\ird{V_{\lambda}(h+V_{\lambda})^\frac{2-q}{q-1}}\right)\\
&=\left(\frac{1-q}{q}\right)^\frac{1}{q-1}\cdot\frac{1}{1-q}\ird{h(h+V_{\lambda})^\frac{2-q}{q-1}}\ge 0.
\end{aligned}
\end{equation*}
(2)Notice that 
\begin{equation*}
\mathcal F[n]-\mathcal F[N_h]=\frac{1}{1-q}\ird{q N_h^{q-1}(n-N_h)-(n^q-N_h^q)}
\end{equation*}
and from the inequality
\begin{equation*}
q b^{q-1}(a-b)-(a^q-b^q)\ge 0
\end{equation*}
for any $a, b\ge 0$, we get the result.
\end{proof}

\subsection{Comparison principle.}

Before proving the main theorem, we still need that
\begin{lem}
For any two non-negative solutions $n_1$ and $n_2$ of equation \eqref{fd1}, defined on a time interval $[0,T]$ with initial data in $L^1(\R^d)$, and any two times $t_1$ and $t_2$ such that $0 \le t_1 \le t_2 \le T$, we have
\[
\int_{\R^d}|n_1(t_2) -n_2(t_2)| dx \le\int_{\R^d}|n_1(t_1) -n_2(t_1)| dx 
\]
and even stronger
\[
\int_{\R^d}|n_1(t_2) -n_2(t_2)|_+ dx \le\int_{\R^d}|n_1(t_1) -n_2(t_1)|_+ dx .
\]
\end{lem}
The lemma implies 
\begin{lem}
(Comparison principle) For any two non-negative solutions $n_1$ and $n_2$ of equation \eqref{fd1} on $[0, T ), T > 0$, with initial data satisfying $n_{0,1} \le n_{0,2} $ almost everywhere, $n_{0,2} \in L^1_{loc} (\R^d)$, then we have $n_1(t) \le n_2(t)$ for almost every $t \in [0,T)$.
\end{lem}
The proof can be found in \cite{herrero1985cauchy, blanchet2009asymptotics, bonforte2006global}.

\section{External drift type equation: convergence without Rate}
\label{Sec3:fastdiffu2}
In this section we mainly prove convergence without rate, which will allow us to assume that $|h_1 -h_*|$ and $|h_2-h_*|$ is arbitrarily small. 
Define relative entropy by
\begin{equation*}
\mathcal{F}[ n| N_{h_*}]= \ird{\phi(n)-\phi(N_{h_*}) - \phi'(n)(n-N_{h_*})}, \quad \phi(x):= \frac {x^q} {q-1} 
\end{equation*}
define 
\[
w:=\frac {n} {N_{h_*}}
\]
then we have
\begin{equation}
\label{eqtran1}
\begin{aligned}
w_t= \frac 1 {N_{h_*}}n_t &= \frac 1 {N_{h^*}} \nabla(q n ^{q-1} \nabla n +n \nabla V_{\lambda})   \\
&= \frac 1 {N_{h_*}} \nabla(q (N_{h_*} w)^{q-1} \nabla (N_{h_*} w) + N_{h_*} w\nabla V_{\lambda}) \\
& =  \frac 1 {N_{h_*}} \nabla \left[N_{h_*} w \left(q w^{q-2}N_{h_*}^{q-1}\nabla w + q w^{q-1} N_{h_*}^{q-2} \nabla N_{h_*} + \frac q {1-q}\nabla (N_{h_*}^{q-1}) \right) \right]\\
& =  \frac 1 {N_{h_*}} \nabla\left[N_{h_*} w \left(\frac {q} {q-1} N_{h_*}^{q-1}\nabla (w^{q-1}) + \frac q {q-1} w^{q-1} \nabla( N_{h_*}^{q-1}) + \frac q {1-q}\nabla (N_{h_*}^{q-1}) \right)\right]\\
& =  \frac 1 {N_{h_*}} \nabla \left[N_{h_*} w \nabla\left(\frac {q} {q-1} N_{h_*}^{q-1}( w^{q-1} -1)\right) \right]
\end{aligned}
\end{equation}
recall that in the third line we use that
\[
\nabla V_{\lambda} = \frac q {1-q}\nabla (N_{h_*}^{q-1})
\]
by homogeneity of $\phi$ the relative entropy we rewrite that
\[
\mathcal{F}[ n| N_{h_*}]= \ird{\phi(w)-\phi(1) - \phi'(1)(w-1)},  \quad w =\frac {n} {N_{h_*}}
\]
so we define the relative entropy
\begin{equation*}
\mathcal{F}[w] = \frac {1} {1-q}  \ird{ \left[ (w-1) - \frac 1 q (w^q-1) \right] N_{h_*}^q }
\end{equation*}
and the relative Fisher information
\[
\mathcal{I}[w] = \frac {q} {(1-q)^2}  \ird{ \left| \nabla \left[ (w^{q-1}-1)N_{h_*}^{q-1} \right]  \right|^2 wN_{h_*} }
\]
it's easily seen that
\[
\mathcal{F}[w] =\frac 1 q \mathcal{F}[ n| N_{h_*}],\quad \mathcal{I}[w] =\frac 1 q \mathcal{I}[ n| N_{h_*}]
\]
and
\[
\frac d {dt} \mathcal{F} [w(t)]=-\mathcal{I} [w(t)]
\]
we omit the regularity here, see Proposition 2 in \cite{blanchet2009asymptotics} for more details. Define 
\[
W_0:=\inf_{x \in \R^N} \frac {N_{h_2}(x)} {N_{h_*}(x)} = \left(\frac {h_*} {h_2}\right)^{\frac{1} {1-q}}<1, \quad W_1:=\inf_{x \in \R^N} \frac {N_{h_1 }(x)} {N_{h_*}(x)} = \left(\frac {h_*} {h_1}\right)^{\frac{1} {1-q}}>1
\]
with such notations, we can rewrite the assumptions as follows:\\
\textbf{(H1')}$n_0$ is a non-negative function in $L^1_{loc}(\R^d)$ and there exist positive constants $h_1<h_2$ such that
\[
0<W_0 \le \frac {N_{h_2}(x)} {N_{h_*}(x)}    \le w(x) \le \frac {N_{h_1}(x)} {N_{h_*}(x)} \le W_1 \le \infty,\quad \forall x \in \R^d
\]
\begin{lem}
(Uniform $C^k$ regularity) Let $q \in(0, 1 )$ and $w \in L^\infty_{loc} ((0,T) \times \R^d )$ be a solution of the nonlinear equation. Then for any $k \in \mathbb N$ and $t_0 \in (0,T)$,
\[
\sup_{t \ge t_0 }\Vert w(t) \Vert_{C^k(\R^d)} <+\infty.
\]
\end{lem}
\begin{proof}
See \cite {blanchet2009asymptotics} Theorem 4.
\end{proof}

\begin{lem}\label{lem3.2}
If $q \in (0,1)$, $w$ satisfies \textbf{(H1')} above, then we have
\[
\frac 1 2 W_1^{q-2} \ird{|w-1|^2 N_{h_*}^q}  \le \mathcal{F} [w]  \le \frac 1 2 W_0^{q-2} \ird{ |w-1|^2 N_{h_*}^q}
\]
\end{lem}
\begin{proof}
For some $a>0$ to be fixed later, we define
\[
\phi_a(w):= \frac {1} {1-q}  \left[ (w-1) - \frac 1 q (w^q-1) \right] -a(w-1)^2
\]
we compute
\[
\phi_a(w) = \frac 1 {q-1}[1-w^{q-1}]-2a(w-1), \quad \phi''_a(w)=w^{q-2}-2a.
\]
Note here $\phi_a(1)=\phi'_a(1)=0 $, recall that $0<W_0<1<W_1 $ and $w \in [W_0, W_1]$. So let $a = W_1^{q-2}/2$, then 
\[
\phi''_a(w)>0, \quad w\in (W_0, W_1),
\] 
which implies
\[
\phi_a(w) \ge 0, \quad w\in (W_0, W_1),
\] 
so the lower bound is proved after multiplying $N_{h_*}^q$ and integrating over $\R^d$. Similarly taking $a = W_0^{q-2}/2$ we can prove the upper bound.
\end{proof}
\begin{cor}
If $w_0$ satisfies \textbf{(H1')}, then the free energy $\mathcal{F}[w(t)]$ is finite for all $t\ge 0$.
\end{cor}
\begin{proof}
By Lemma~\ref{lem3.2}, we have
\[
\frac 2 {W_0^{q-2}} \mathcal{F}[w] \le \frac 2 {W_0^{q-2}} \mathcal{F}[w_0] \le \ird{|w-1|^2 N_{h_*}^ {q-2} } \le \ird{|w-1| |N_{h_2} -N_{h_1}| N_{h_*}^ {q-2}}
\]
and it is easily seen that 
\[
|N_{h_2} -N_{h_1}|  \le C_*(1+|x|)^{-\frac {\lambda(2-q)} {1-q}}
\]
for some constant $C_*>0$. So the proof is concluded since $w-1$ is integrable and $|N_{h_2} -N_{h_1}| N_{h_*}^ {q-2}$ is bounded.
\end{proof}
\begin{lem}
Let $q \in (0, 1)$. If $w$ is a solution and $w_0$ satisfying \textbf{(H1')}, then
\[
\lim_{t \to \infty}  w(t, x ) = 1, \quad \forall x \in \R^d.
\]
\end{lem}
\begin{proof}
Let $w_\tau(t, x)=w(t+\tau, x)$. Since the functions are uniformly $C^1$ continuous, we have there exists a sequence $\tau_n \to \infty$ such that $w_{\tau_n}$ converges to a function $w_{\infty}$ uniformly in every compact set.  By interior regularity of the solutions, the derivatives also converge everywhere. Since $w(x) \ge W_0$ we have $w_\infty \ge W_0>0$.\\
Since $\mathcal{F}[w]$ is finite, we compute
\[
\mathcal{F}[w(\tau_n)] - \mathcal{F}[w(\tau_n+1)]  = \int_{0}^1\mathcal{I}[w(t + \tau_n)] dt
\]
by Fatou's Lemma we have
\[
\int_{0}^1\mathcal{I}[w_\infty] dt \le \lim_{n \to \infty} \int_{0}^1\mathcal{I}[w(t + \tau_n)] dt =  \lim_{n \to \infty} \mathcal{F}[w(\tau_n)] - \mathcal{F}[w(\tau_n+1)]=0
\]
which is 
\[
\int_0^1 \ird{ \left|\nabla\left[(w_{\infty}^{p-1}(t, x) -1 ) N_{h_*}^{q-1}(x)  \right]       \right|^2 w_{\infty}(t, x)  N_{h_*}(x)   }dt =0
\]
since $w_\infty \ge W_0$, this implies that $w_\infty $ is constant.  By the conservation of mass, we deduce $w_\infty =1$. Since the limit is unique, the whole $w(t)$ converges to $1$ as $t \to \infty$.
\end{proof}
\begin{cor}
Let $q \in (0, 1)$. If $w$ is a solution of \eqref{eqtran1} and $w_0$ satisfying \textbf{ (H1')}, then
\[
\lim_{t \to \infty}  \Vert w(t) -1 \Vert_{L^\infty}=0.
\]
\end{cor}
\begin{proof}
First we compute
\[
|w(t, x) - 1| = |n(t, x) - N_{h_*}|\cdot |N_{h_*}^{-1} | \le |N_{h_1} -N_{h_0}|  \cdot|N_{h_*}^{-1} | \le \mathcal C(1+|x|)^{-\frac {\lambda} {1-q}} \in L^p,
\]
for some constant $\mathcal C>0$ and some $p$ large. By dominated convergence theorem, we have
\[
\lim_{t \to \infty}  \Vert w(t) -1 \Vert_{L^p}=0.
\]
by the inequality in \cite{catrina2001caffarelli, nirenberg1959on}
\[
\Vert f \Vert_{L^\infty} \le \Vert f \Vert_{C^1}^{q}  \Vert f \Vert_{L^p}^{1-q}
\]
for $q=\frac {p} {p+N}$, we deduce the result.
\end{proof}

\section{External drift type equation: convergence with rate}
\label{Sec4:fastdiffu3}
In this section we prove the convergence with rate around steady state, together with the convergence without rate in the former section we are able to give the convergence rate for the equation \eqref{fd1}. The fact that $w(t)$ converges uniformly to $1$ as $t \to \infty$ allows us to improve the lower and upper bounds $W_0$ and $W_1$ for the function w(t), at the price of waiting some time. For any $\epsilon> 0$ there exists a time $t_0 = t_0(\epsilon)$ such that
\[
1-\epsilon \le w(t, x) \le 1+\epsilon, \quad \forall(t, x) \in (t_0, \infty) \times \R^N.
\]
define the function
\begin{equation}
\label{weight2}
\mathcal M_1(x) := \Delta(N_{h_*}^{q-1}) \mathcal M +\nabla  \mathcal M  \cdot \nabla(N_{h_*}^{q-1})
\end{equation}
remind that
\[
\frac{q}{1-q}N_{h_*}^{q-1}(x)=h_*+V_{\lambda}(x)
\]
so
\begin{equation}
\label{weight5}
\mathcal M_1(x):=\left\{
\begin{aligned}
  & \frac {1-q}{q}\cdot\frac{N|x|^{2\lambda-4}+(\lambda+N-2)|x|^{\lambda-2}}{(1+|x|^{\lambda-2})^2} , &\quad \lambda \ne 2 \\
 & \frac{N(1-q)}{q}, &\quad \lambda =2
\end{aligned}
\right.
\end{equation}
it is easily seen that for any $x\in\mathbb R^N$,
\[
|\mathcal M_1(x)| \le \frac{1-q}{q}\mbox{max}\{N,\lambda+N-2\} ,
\]
\begin{lem}\label{lem4.1}
For any smooth function $\alpha(x)$ and the functions $\mathcal M(x), \mathcal M_1(x)$ defined above, we have
\begin{equation*}
\begin{aligned}
&\ird{|\nabla(\alpha(w) N_{h_*}^{q-1})|^2N_{h_*} \mathcal M    }\\
=&\ird{|\alpha'(w)|^2|\nabla w|^2N_{h_*}^{2q-1}\mathcal M   }+ \frac {1} {1-q}\ird{\alpha^2(w) \nabla|  (N_{h_*}^{q-1})|^2N_{h_*} \mathcal M  }-\ird{ \alpha^2(w)  {N_{h_*}^{q} }\mathcal M  _1 }.
\end{aligned}
\end{equation*}
\end{lem}

\begin{proof} We have
\begin{equation*}
\begin{aligned}
&\ird{|\nabla(\alpha(w) N_{h_*}^{q-1})|^2N_{h_*} \mathcal M     }\\
=&\ird{|N_{h_*}^{q-1}\nabla \alpha(w)  + \alpha(w) \nabla  (N_{h_*}^{q-1})   |^2N_{h_*} \mathcal M  }\\
=&\ird{|\alpha'(w)|^2|\nabla w|^2N_{h_*}^{2q-1}\mathcal M   }+\ird{\alpha^2(w)|\nabla  (N_{h_*}^{q-1})|^2N_{h_*} \mathcal M   }+\ird{ \nabla(\alpha^2(w) ) N_{h_*}^{q} \mathcal M  \nabla(N_{h_*}^{q-1})  }\\
=&\ird{|\alpha'(w)|^2|\nabla w|^2N_{h_*}^{2q-1} \mathcal M   }+\ird{\alpha^2(w)|\nabla  (N_{h_*}^{q-1})|^2N_{h_*} \mathcal M  }-\ird{ \alpha^2(w)  \nabla ({N_{h_*}^{q} })\nabla(N_{h_*}^{q-1}) \mathcal M  }\\
&-\ird{ \alpha^2(w)  {N_{h_*}^{q} }\Delta(N_{h_*}^{q-1}) \mathcal M   }-\ird{ \alpha^2(w)  {N_{h_*}^{q} }\nabla  \mathcal M   \cdot \nabla(N_{h_*}^{q-1}) }\\
=&\ird{|\alpha'(w)|^2|\nabla w|^2N_{h_*}^{2q-1}\mathcal M   }+ \frac {1} {1-q}\ird{\alpha^2(w)|\nabla  (N_{h_*}^{q-1})|^2N_{h_*} \mathcal M   }-\ird{ \alpha^2(w)  {N_{h_*}^{q} }\mathcal M  _1 }
\end{aligned}
\end{equation*}
where in the third line we use that
\[
\nabla ({N_{h_*}^{q} }) =\frac q {q-1} N_{h_*} \nabla ({N_{h_*}^{q-1} }).
\]
\end{proof}
Next, we define two functionals
\begin{equation}
\label{functional1}
\Phi_1[f]  := \frac 1 2 \ird{|f|^2 N_{h_*}^{2-q} }
\end{equation}
and
\begin{equation}
\label{functional2}
\Phi_2[f]  := q \ird{| \nabla f|^2 N_{h_*} \mathcal M    }
\end{equation}
\begin{lem}\label{lemhp}(Hardy-Poincar\'e inequality) Define $q_* = \frac {N-2-\lambda} {N-2}$.  For any $ 1 \le N \le 2$, $q\in(0,1)$ or $N \ge 3$, $q\in (0, 1), q \neq q_*$ and any suitable function  $f$ satisfies
\[
\ird{fN_{h_*}^{2-q}} =0
\]
we have
\[
\Phi_1[f] \le C_{q, N,\lambda} \Phi_2[f]  
\]
for some constant $C_{q, N,\lambda}>0$.
\end{lem}
Since $\frac {1}{C_\lambda} (1+|x|^2)(1+|x|^{\lambda-2}) \le 1+|x|^\lambda \le C_\lambda (1+|x|^2)(1+|x|^{\lambda-2})$, for some $C_\lambda>0$ so this is just the classical Hardy-Poincar\'e inequality, see for example \cite{blanchet2009asymptotics, blanchet2007hardy} for the full proof.

\begin{lem}\label{lem4.3}
Let $w$ be the solution to the equation \eqref{eqtran1} and denote $\eta:=(w-1)N_{h_*}^{q-1}$. There exist constants $\beta_1,\beta_2>0$, such that
\begin{equation*}
\Phi_2[\eta]\le \beta_1\mathcal{I}_1[w]+\beta_2 \Phi_1[\eta].
\end{equation*}
with
\[
\mathcal{I}_1[w]:= q\ird{|\nabla(h_q(w)N_{h_*}^{q-1})|^2N_{h_*} \mathcal M   w }
\]
here $\beta_2$ can be arbitrarily small if $t$ large.
\end{lem}
\begin{proof}
Let $\alpha_0 = W_0^{2(2-q)} , \alpha_1 = W_1^{2(2-q)} $. Since $|h_2/h_m|$ is non-decreasing, we have
\[
\alpha_0 \le \left|\frac {h_2(W_0) } {h_q(W_0)} \right|^2 \le \left|\frac {h_2(w) } {h_q(w)} \right|^2  \le \left|\frac {h_2(W_1) } {h_q(W_1)} \right|^2 
  \le \alpha_1
\]
and
\[
\alpha_0 \le \left|\frac {h'_2(w) } {h'_q(w)} \right|^2  \le \alpha_1
\]
note that $\alpha_0 = \alpha_0(W_0) <1 < \alpha_1 = \alpha_1(W_1)$ and both converges to 1 as $W_0, W_1 \to 1$. 
By Lemma \ref{lem4.1}, take $\alpha(w) = h_2(w)$ we have
\begin{equation*}
\begin{aligned}
& \Phi_1[\eta]=q\ird{|\nabla(h_2(w)N_{h_*}^{q-1})|^2N_{h_*} \mathcal M }\\
=&q\ird{|h'_2(w)|^2|\nabla w|^2N_{h_*}^{2q-1} \mathcal M  }+ \frac {q} {1-q}\ird{h_2^2(w)|\nabla  (N_{h_*}^{q-1})|^2N_{h_*} \mathcal M }-q\ird{ h_2^2(w)  {N_{h_*}^{q} }\mathcal M_1 }\\
\le&q \alpha_1\ird{|h'_q(w)|^2|\nabla w|^2N_{h_*}^{2q-1} \mathcal M }+ \frac {\alpha_1 q} {1-q}\ird{h^2_q(w)|\nabla  (N_{h_*}^{q-1})|^2N_{h_*} \mathcal M }-q\ird{ h^2_2(w)  {N_{h_*}^{q} }\mathcal M_1  }
\end{aligned}
\end{equation*}
and take $\alpha(w) = h_q(w)$ in Lemma \ref{lem4.1}, we have
\begin{equation*}
\begin{aligned}
&\ird{|h'_q(w)|^2|\nabla w|^2N_{h_*}^{2q-1} \mathcal M   }\\
=&\ird{|\nabla(h_q(w)N_{h_*}^{q-1})|^2N_{h_*} \mathcal M   }- \frac {1} {1-q}\ird{h^2_q(w)|\nabla  (N_{h_*}^{q-1})|^2N_{h_*} \mathcal M  }+\ird{ h^2_q(w)  {N_{h_*}^{q} }\mathcal M_1 }
\end{aligned}
\end{equation*}
so
\begin{equation*}
\begin{aligned}
\Phi_1[\eta] &\le q\alpha_1\ird{|\nabla(h_q(w)N_{h_*}^{q-1})|^2N_{h_*} \mathcal M   }+q\ird{(\alpha_1 |h_q(w)|^2-|h_2(w)|^2)   \mathcal M  _1   N_{h_*}^{q} }\\
&\le q\alpha_1\ird{|\nabla(h_q(w)N_{h_*}^{q-1})|^2N_{h_*} \mathcal M   }+q\ird{\left(\frac {\alpha_1} {\alpha_0} -1
\right) |h_2(w)|^2 |\mathcal M_1 | N_{h_*}^{q} }
\end{aligned}
\end{equation*}
next, since  $0< W_0 \le w$,
\begin{equation*}
q\ird{|\nabla(h_q(w)N_{h_*}^{q-1})|^2N_{h_*} \mathcal M }\le \frac{1}{W_0}\mathcal{I}_1[w]
\end{equation*}
recall that
\begin{equation*}
\eta=(w-1)N_{h_*}^{q-1}, \quad | \mathcal M_1 | \quad\mbox{is uniformly bounded},
\end{equation*}
so we have
\[
q\ird{\left(\frac {\alpha_1} {\alpha_0} -1\right) |h_2(w)|^2  N_{h_*}^{q} |\mathcal M_1| } = q\left(\frac {\alpha_1} {\alpha_0} -1\right)\ird{|\eta|^2 N_{h_*}^{2-q} |\mathcal M_1| }  \le \mathcal C' q\left(\frac {\alpha_1} {\alpha_0} -1\right) \mathcal F[\eta] 
\]
for some constant $\mathcal C'>0$, since $\frac {\alpha_1} {\alpha_2} -1 \to 0$ as $t \to \infty$, we finally we obtain the result.
\end{proof}

\begin{cor}
With the notations above, we have
\[
\mathcal{F} [w] \le \gamma \mathcal{I}_1 [w]
\]
for some $\gamma>0$.
\end{cor}
\begin{proof}
By Lemma \ref{lem4.3},
\begin{equation*}
\Phi_2[\eta]\le \beta_1\mathcal{I}_1[w]+\beta_2 \Phi_1[\eta].
\end{equation*}
and since
\[
\ird{\eta N_{h_*}^{2-q}} = \ird{(w-1) N_{h_*}} = \ird{n-N_{h_*}} =0
\]
by Hardy Poincar\'e inequality
\[
\Phi_1[\eta] \le C_{q, N,\lambda} \Phi_2[\eta]  
\]
 we have
 \[
\Phi_1[\eta] \le\frac { \beta_1C_{q, N,\lambda } }  {1- \beta_2 C_{q, N,\lambda}} \mathcal{I}_1[w]
\]
since we can pick $\beta_2$ small. So the theorem ends since
\[
\Phi_1[\eta] \ge  {W_0}^{2-q} \mathcal{F} [w]
\]
from Lemma \ref{lem3.2}.
\end{proof}
\begin{cor}
\label{cor4.5}
There exist constants $\mathcal K, \beta>0$, such that\\

(i)
For any $\lambda\ge 2$,
\[
\mathcal{F}[w] \le \mathcal K e^{-\beta t}\quad\mbox{for any}\quad t\ge 0;
\]

 (ii)
For any $\lambda \in(0, 2)\,\,and\,\,\frac {N+2} {N+2+\lambda}<q<1$,
\[
\mathcal{F}[w] \le \mathcal K(1+t)^{-\beta}\quad\mbox{for any}\quad t\ge 0.
\]
\end{cor}\begin{proof}
For $\lambda \ge 2$, $\mathcal M(x) \le 1$, so 
\[
\mathcal{F} [w] \le \mathcal K\mathcal{I}_1 [w] \le \mathcal K\mathcal{I} [w]
\]
so the conclusion follows. For $\lambda \in(0, 2)$,  by H\"older inequality
\[
\mathcal{F} [w] \le \gamma \mathcal{I}_1 [w] \le \mathcal K \mathcal{I} [w]^{\alpha}\mathcal{I}_2[w]^{1-\alpha}
\]
for some $\alpha \in (0, 1)$ with
\begin{equation}
\label{ratecon1}
\mathcal{I}_2[w]= q\ird{|\nabla(h_q(w) N_{h_*}^{q-1})|^2N_{h_*} (1+|x|^\delta)w }
\end{equation}
for some $\delta> 2-\lambda$. Recall that
\[
\sup_{t \ge t_0 }\Vert w(t) \Vert_{C^k(\R^N)} <+\infty
\]
so \eqref{ratecon1} is the same as
\[
\ird{|N_{h_*}^{q-1}|^2N_{h_*} (1+|x|^{\delta})  } < +\infty
\]
which is equivalent to $q > \frac {N+2} {N+2+\lambda}$.

\end{proof}

Theorem~\ref{thm1.1} can be directly deduced by Corollary~\ref{cor4.5} and Lemma~\ref{lem3.2}.

\section{Mean-field type equation: preliminaries}
\label{Sec5:nonlidiffu1}
In this section, we give some basic properties of the equation \eqref{fd2}. Remind that from now on we always suppose that $\lambda\ge 2$ and $q\in(q_*(N,\lambda),1)$, here
\[
q_*(N,\lambda)=\left\{
\begin{aligned}
  & \frac{2N}{2N+\lambda} , &\quad \lambda > 2 \\
  & \frac{N}{N+2}, &\quad \lambda =2
\end{aligned}
\right.
\]
we first define the change of variable
\begin{equation}
\label{cnange2}
\rho = \rho_\infty v = \rho_\infty (1+g) = \rho_\infty + j
\end{equation}
the existence of minimizers of the free energy is proved in \cite{dolbeault2018reverse}. To continue our proof, we still need the following theorems, the proofs  are similar as the fast diffusion equation with external drift \eqref{fd1} and thus omitted.

\begin{lem}
(Comparison principle) For any two non-negative solutions $\rho_1$ and $\rho_2$ of the equation \eqref{fd2} on $[0, T ), T > 0$, with initial data satisfying $\rho_{0,1} \le \rho_{0,2} $ almost everywhere, $\rho_{0,2} \in L^1_{loc} (\R^d)$, then we have $\rho_1(t) \le \rho_2(t)$ for almost every $t \in [0,T)$.
\end{lem}

\begin{lem}
(Uniform $C^k$ regularity) Let $q \in(q_*(N,\lambda), 1 )$ and $u$ be a solution of equation \eqref{fd2}. Then $v=\frac{\rho}{\rho_{\infty}} \in L^\infty_{loc} ((0,T) \times \R^N )$ satisfies for any $k \in \mathbb N$ and any $t_0 \in (0,T)$,
\[
\sup_{t \ge t_0 }\Vert v(t) \Vert_{C^k(\R^N)} <+\infty.
\]
\end{lem}
We can also similarly define
\[
\mathcal W_0:=\inf_{x \in \R^N}\frac{\rho_2(x)}{\rho_*(x)}, \quad\mathcal W_1:=\inf_{x \in \R^N}\frac{\rho_1(x)}{\rho_*(x)}
\]
\begin{lem}\label{lem5.3}
If $q \in (q_*(N,\lambda),1)$, $0 < \mathcal W_0 \le v \le \mathcal W_1 < +\infty$, then we have
\[
\frac 1 2 \mathcal W_1^{q-2} \ird{|v-1|^2 \rho_{\infty}^q}  \le  -\frac{1}{1-q}\ird{\rho_\infty^q [v^q -1 - q (v -1) ] } \le \frac 1 2 \mathcal W_0^{q-2} \ird{ |v-1|^2 \rho_{\infty}^q}
\]
\end{lem}
\begin{lem}
\label{lem5.4}
Let $q\in (q_*(N,\lambda), 1)$. If $\rho$ is a solution of \eqref{fd2} satisfying (H2), then for $v=\frac{\rho}{\rho_{\infty}}$,
\[
\lim_{t \to \infty}  \Vert v(t) -1 \Vert_{L^\infty}=0.
\]
\end{lem}
Remind that now $q>q_*(N,\lambda)>q_{\#}$, all the lemmas can be proved as the simple case of fast diffusion equation. So from now on, we focus now on the convergence with rate. The fact that $v(t)$ converges uniformly to $1$ as $t \to \infty$ allows us to improve the lower and upper bounds $\mathcal W_0$ and $\mathcal W_1$ for the function $v(t)$, at the price of waiting some time. For any $\epsilon> 0$ there exists a time $t_0 = t_0(\epsilon)$ such that
\[
1-\epsilon \le v(t, x) \le 1+\epsilon, \quad \forall(t, x) \in (t_0, \infty) \times \R^d.
\]
Recall that $\rho$ is the solution of \eqref{fd1}. Then $v=\frac{\rho}{\rho_{\infty}}$ satisfies
\begin{equation}
\label{changev1}
\begin{aligned}
v_t= \frac 1 {\rho_\infty}   \rho_t &= \frac 1 {\rho_\infty} \nabla(q \rho ^{q-1} \nabla \rho +\rho \nabla V_\lambda*\rho)   \\
&= \frac 1 {\rho_\infty} \nabla \left( q (\rho_\infty v)^{q-1} \nabla (\rho_\infty v) + \rho_\infty v\nabla V_\lambda*(\rho_\infty v)    \right) \\
& =  \frac 1 {\rho_\infty} \nabla \left[\rho_\infty v \left(\frac {q} {q-1} \rho_\infty^{q-1}\nabla (v^{q-1}) + \frac {q} {q-1}v^{q-1} \nabla(\rho_\infty^{q-1}) +\nabla V_\lambda*(\rho_\infty v)   \right)   \right] \\
& =  \frac 1 {\rho_\infty} \nabla \left[\rho_\infty  v \nabla\left(\frac {q} {q-1} \rho_\infty^{q-1}v^{q-1}  +  V_\lambda*(\rho_\infty v) \right) \right]
\end{aligned}
\end{equation}
the associated free energy  and Fisher information defined in \eqref{free5}, \eqref{free6} become
\begin{equation}
\label{free2}
\mathbb{F}[v]:=-\frac{1}{1-q}\ird{\rho_\infty^q [v^q -1 - q (v-1)] }+\frac 1 2\ire{\ird{ V_{\lambda}(x-y)\rho_\infty(x)g(x)\rho_\infty(y)g(y)}}
\end{equation}
we note here if $v=1$, then $\mathbb F[v] =0$, in other words, $\mathbb F[v] = \mathbb F[\rho|\rho_\infty]$, then we define
\begin{equation}
\label{fisher2}
\mathbb{I}[v]:=\ird{\rho_\infty v \left|\frac{q}{1-q}\nabla(\rho_\infty ^{q-1}(v^{q-1} -1))- \nabla V_\lambda*(\rho_\infty g)\right|^2}
\end{equation}
we similarly define the function
\begin{equation}
\label{weight2}
\mathcal  N_1(x) := \Delta(\rho_\infty^{q-1}) \mathcal M +\nabla  \mathcal M  \cdot \nabla(\rho_\infty^{q-1})
\end{equation}
define three functionals $\Psi_1(g), \Psi_2(g), \Psi_3(g)$ by
\[
\Psi_1(g) :=\ird{\rho_\infty^q g^2}
\]
\[
\Psi_2(g) := q^2\ird{\rho_\infty|\nabla(\rho_\infty^{q-1}g)|^2 \mathcal M}
\]
and
\[
\Psi_3(g) :=  \sum_{i=1}^N\left(\ird{x_i g(x)\rho_{\infty}(x) }\right)^2
\]
note that $\Psi_1(g), \Psi_2(g)$ are actually the similar forms as \eqref{functional1}, \eqref{functional2} from Section~\ref{Sec4:fastdiffu3}. 

 Suppose also that $\mathcal W_0 \le v\le \mathcal W_1$ then we have
\begin{lem}\label{lem5.5}
For some $\epsilon>0$ small, we have
\begin{equation}
\label{keyes2}
\ird{\rho_\infty \left|\nabla \left(\rho_\infty ^{q-1}\frac {(v^{q-1} -1)} {q-1}  - \nabla(\rho_\infty^{q-1}(v-1))   \right)\right|^2 \mathcal M} \le \epsilon \Psi_1(g) + \epsilon \Psi_2(g) 
\end{equation}
and $\epsilon>0$ can be arbitrarily small if $\mathcal W_0$, $\mathcal W_1$ is close enough to $1$.
\end{lem}
\begin{proof}

We first show that there exists a constant $\Theta(N,\lambda)>0$, such that for any $x\in\mathbb R^N$, $|\mathcal N_1(x)| \le \Theta$. 
The case that $\lambda=2$ is simple. From \eqref{case2}, we have
\[
\mathcal N_1(x)=\Delta(\rho_{\infty}(x)^{q-1})=\frac{1-q}{2q}\Delta(|x|^2)=\frac{N(1-q)}{q}<2
\]
for $\lambda>2$, notice that 
\[
\frac{q}{1-q}\mathcal N_1(x)=\Delta V_{\lambda}*\rho_{\infty}(x)\cdot \mathcal M(x)+\nabla(V_{\lambda}*\rho_{\infty})(x)\nabla\mathcal M(x)
\]
remind that
\[
[\nabla\mathcal M(x)]_i=\frac{(2-\lambda)x_i|x|^{\lambda-4}}{(1+|x|^{\lambda})^2}
\]
using $\ird{\rho_{\infty}}=1$, we have
\[
|\Delta V_{\lambda}*\rho_{\infty}(x)\cdot \mathcal M(x)|\le \Delta V_{\lambda}(x)\mathcal M(x)=\frac{(\lambda+N-2)|x|^{\lambda-2}}{1+|x|^{\lambda-2}}\le\lambda+N-2
\]
and
\[
|\nabla(V_{\lambda}*\rho_{\infty})(x)\nabla\mathcal M(x)|\le \frac{(\lambda-2)|x|^{2\lambda-4}}{(1+|x|^{\lambda-2})^2}\le \lambda-2
\]
so 
\[
|\mathcal N_1(x)|\le\frac{(1-q)(2\lambda+N-4)}{q}<\frac{\lambda(2\lambda+N-4)}{2N}
\]
next, by Lemma \ref{lem4.1}, for any function $\beta$ we have
\begin{equation*}
\begin{aligned}
&\ird{\rho_\infty \left|\nabla \left(\rho_\infty^{q-1} \beta (v)   \right)\right|^2 \mathcal M }   
\\
=&\ird{|\beta'(v)|^2|\nabla v|^2 \rho_\infty^{2q-1}   \mathcal M }+ \frac {1} {1-q}\ird{|\beta(v)|^2 |\nabla ( \rho_\infty^{q-1})|^2\rho_\infty \mathcal M}-\ird{ |\beta(v)|^2  {\rho_\infty^{q} } \mathcal N_1}
\end{aligned}
\end{equation*}
take $\beta= h_q(v) -h_2(v)$ and $\beta=h_2(v)$ separately, and recall that 
$h_k(v):=\frac{v^{k-1}-1}{k-1}$, we have
\begin{equation*}
\begin{aligned}
&\ird{\rho_\infty \left|\nabla \left(\rho_\infty^{q-1} h_q(v) -h_2(v)   \right)\right|^2 \mathcal M }   
\\
\le&\ird{|h'_q(v) -h'_2(v) |^2|\nabla v|^2 \rho_\infty^{2q-1}  \mathcal M }+ \frac {1} {1-q}\ird{|h_q(v) -h_2(v) |^2 |\nabla ( \rho_\infty^{q-1})|^2\rho_\infty  \mathcal M} \\
&-\ird{ |h_q(v) -h_2(v) |^2  {\rho_\infty^{q} } \mathcal N_1}\\
\le&\epsilon \ird{|h'_2(v) |^2|\nabla v|^2 \rho_\infty^{2q-1} \mathcal{M} }+\epsilon \frac {1} {1-q}\ird{|h_2(v) |^2 |\nabla ( \rho_\infty^{q-1})|^2\rho_\infty  \mathcal{M} } \\
&-\ird{ |h_q(v) -h_2(v) |^2  {\rho_\infty^{q} } \mathcal N_1}\\
\le&\epsilon \ird{\rho_\infty \left|\nabla \left(\rho_\infty^{q-1} h_2(v)   \right)\right|^2 \mathcal M } + \epsilon\ird{ |h_2(v)|^2  {\rho_\infty^{q} } |\mathcal N_1| }\le \epsilon \Psi_1(g) + \epsilon \Psi_2(g) 
\end{aligned}
\end{equation*}
the lemma is thus proved.
\end{proof}

\section{Mean-field type equation: analysis of $\lambda =2$}
\label{Sec6:nonlidiffu2}
We first study the quadratic forms associated with the expansion of the $\mathbb F$ and $\mathbb I$ around $\rho_{\infty}$. For a smooth perturbation $g$ 
of $\rho_{\infty}$ such that $\ird{g\rho_{\infty}(x)} =0$, define
\begin{equation}
\label{qua21}
\begin{aligned}
\mathbb Q_1[g]:&=\lim_{\varepsilon\to 0}\frac{2}{\varepsilon^2}(\mathbb F[\rho_{\infty}(1+\varepsilon g)]-\mathbb F[\rho_{\infty}])\\
&=q\ird{\rho_\infty^q g^2}+\ire{\ird{ V_{\lambda}(x-y)\rho_\infty(x) g(x)\rho_\infty(y) g(y)}}
\end{aligned}
\end{equation}
and
\begin{equation}
\label{qua22}
\mathbb Q_2[g]:=\lim_{\varepsilon\to 0}\frac{1}{\varepsilon^2}\mathbb I[\rho_{\infty}(1+\varepsilon g)]=\ird{\rho_ \infty \left| q\nabla(\rho_\infty^{q-1}g)-\nabla V_{\lambda}*(\rho_\infty g))\right|^2}
\end{equation}
in this section, we focus on the case $\lambda=2, \frac{N}{N+2}<q<1$ and prove the result about large time asymptotic behavior of equation \eqref{fd2}.
\subsection{The coercivity result}We first prove that
\begin{lem}
When $\lambda=2$, for $\mathbb Q_1[g]$ and $\mathbb Q_2[g]$ we have
\[
\mathbb Q_1[g] = \Psi_1(g)-\Psi_3(g)
\]
and
\[
\mathbb Q_2[g] = \Psi_2(g)+ 3\Psi_3(g).
\]
\end{lem}

\begin{proof} Recall that
\[
\rho = \rho_\infty v = \rho_\infty (1+g) = \rho_\infty + j
\]
we have
\[
\ird{ g\rho_{\infty}} =0.
\]
For $\lambda=2$, it's easily seen that
\[
\ire{\ird{ V_{\lambda}(x-y)\rho_\infty(x) g(x)\rho_\infty(y) g(y)}} = \frac 1 2\ire{\ird{ |x-y|^2 g\rho_{\infty}(x) g\rho_{\infty}(y)}}  =-\Psi_3( g) 
\]
so the $\mathbb Q_1$ term follows. For the $ \mathbb Q_2[g]$ term
\begin{equation*}
\begin{aligned}
\mathbb Q_2[g]:=&\ird{\rho_\infty\left|q\nabla(\rho_\infty^{q-1}g) -\nabla V_{\lambda}*(g\rho_{\infty})\right|^2}\\
=&q^2\ird{\rho_\infty|\nabla(\rho_\infty^{q-1}g)|^2}-2q\ird{\rho_\infty\nabla(\rho_\infty^{q-1}g)\cdot\nabla V_{\lambda}*(g\rho_{\infty})}\\
&+\ird{\rho_\infty|\nabla V_{\lambda}*(g\rho_{\infty})|^2}
\end{aligned}
\end{equation*}
for the second  term, by integration by parts
\begin{equation*}
\begin{aligned}
&-2q\ird{\rho_\infty\nabla(\rho_\infty^{q-1}g)\cdot\nabla V_{\lambda}*(g\rho_{\infty})} \\
&= 2q\ird { \rho_\infty^{q-1} g \nabla \rho_\infty \cdot \nabla V_{\lambda}*(g\rho_{\infty}) } + 2q\ird {\rho_\infty^{q} g \Delta V_{\lambda}*(g\rho_{\infty})}
\end{aligned}
\end{equation*}
recall that
\[
\nabla (\rho_\infty^{q-1})=  \frac {1-q } {q }\nabla V_\lambda, \quad \rho_\infty \nabla(\rho_\infty^{q-1}) = (q-1) \rho_\infty^{q-1} \nabla \rho_\infty
\]
we have
\begin{equation*}
\begin{aligned}
2q\ird { \rho_\infty^{q-1} g\nabla \rho_\infty \cdot \nabla V_{\lambda}*(g\rho_{\infty}) } =& 2\ird {\frac {q} {q-1} g \rho_\infty \nabla ( \rho_\infty^{q-1})  \cdot \nabla V_\lambda * (g\rho_{\infty}) }\\
=& - 2\ird {  g\rho_{\infty} \nabla V_\lambda  \cdot \nabla V_\lambda * (g\rho_{\infty}) } 
\end{aligned}
\end{equation*}
and we compute
\begin{equation*}
\begin{aligned}
-\ird{g\rho_{\infty}\nabla V_\lambda  \cdot \nabla V_\lambda*(g\rho_{\infty})}&=-\sum_{i=1}^N\ird{\ire{g(x)\rho_{\infty}(x)g(y)\rho_{\infty}(y)x_i(x_i-y_i)}}\\
&=\sum_{i=1}^N\ird{\ire{g(x)\rho_{\infty}(x)g(y)\rho_{\infty}(y)x_i y_i}}\\
&=\sum_{i=1}^N\left(\ird{x_i g(x)\rho_{\infty}(x)}\right)^2
\end{aligned}
\end{equation*}
since the mass of $g\rho_{\infty}$ is 0, we have
\[
2q\ird {\rho_\infty^{q}g \Delta V_{\lambda}*(g\rho_{\infty})} = 0
\]
for the third term we have
\[
\ird{\rho_\infty|\nabla V_{\lambda}*(g\rho_{\infty})|^2}=\ird{\rho_\infty\sum_{i=1}^N\left(\ire{(x_i-y_i)g(y)\rho_{\infty}(y)}\right)^2}=\sum_{i=1}^N\left(\ird{x_i g(x)\rho_{\infty}(x)}\right)^2
\]
we conclude by gathering all the terms.
\end{proof}

\subsection{Large time asymptotic behavior (proof of Theorem~\ref{thm1.3}).}

\begin{cor}
For $\lambda=2$, $q\in (\frac{N}{N+2}, 1)$, there exists $\gamma >0$, such that for $v(t)$ as the solution of \eqref{changev1},
\[
\mathbb{F} [v] \le \gamma \mathbb{I}[v].
\]
\end{cor}
\begin{proof}
Recall
\[
\rho = \rho_\infty v = \rho_\infty (1+g) = \rho_\infty + j
\]
we prove by talking on the relationship between $\mathbb{Q}_1$ and $\mathbb{F}$ and the relationship between $\mathbb{Q}_2$ and $\mathbb{I}$. by Lemma \ref{lem3.2} above we have
\[
\frac q 2 \mathcal W_1^{q-2} \ird{|v-1|^2 \rho_\infty^q }  \le -\frac{1}{1-q}\ird{\rho_\infty^q [v^q -1 - q (v -1) ] } \le \frac q 2 
\mathcal W_0^{q-2} \ird{ |v-1|^2 \rho_\infty^q }
\]
which implies
\[
\mathbb{F}[v] \le2\Psi_1(g)-\frac 1 2\Psi_3(g) 
\]
since 
\[
|a+b|^2 \le2 |a|^2 + 2|b|^2
\]
we have
\[
\mathbb Q_2[g] \le 2\mathbb{I} [v]  + 2q^2\ird{\rho_\infty \left|\nabla \left(\rho_\infty ^{q-1}\frac {(v^{q-1} -1)} {q-1}  - \nabla(\rho_\infty^{q-1}(v-1))   \right)\right|^2 }
\]
and since
\[
\ird{(g\rho_\infty^{q-1}) \rho_{\infty}^{2-q}} = \ird{g \rho_\infty}=0
\]
so we use the Hardy-Poincar\'e inequality (\cite{{blanchet2009asymptotics}}, Appendix A, Theorem 1) on $g\rho_\infty^{q-1}$ to have
\[
\Psi_2(g) \ge\mathcal C_{q, N,2} \Psi_1(g) 
\]
here 
\[
\mathcal C_{q,N,2}=\frac{(N-4-q(N-2))^2}{8q(1-q)}
\]
then by Lemma \ref{lem5.5} we have
\[
\mathbb{I}[v] \ge \frac 1 2 \mathbb Q_2[g] - \epsilon \Psi_2(g)  - \epsilon \Psi_1(g) \ge \frac 1 4 \Psi_2[g]\ge \frac {\mathcal C_{q, N,2}} 4\Psi_1[g] \ge \frac {\mathcal C_{q, N,2}} 8\mathbb F[v]
\]
if we take $\epsilon >0$ small enough. 
\end{proof}

Theorem~\ref{thm1.3} is the direct result by using Gr\"onwall inequality.

\section{Mean-field type equation: analysis of general $\lambda> 2$}
\label{Sec7large}
In this section, we consider the general case $\lambda>2$. We will prove that for $q$ close enough to 1, there exists the large time asymptotic behavior, which 
finishes the proof of Theorem~\ref{thm1.4}.
\label{Sec8:nonlidiffu4}
\subsection{Some important propositions of the stationary solution}
Remind that the stationary $\rho_{\infty}$ satisfies
\begin{equation}
\label{remindsta}
\frac{q}{1-q}\rho_{\infty}^{q-1}=\frac{1}{\lambda}|x|^{\lambda}*\rho_{\infty}+C
\end{equation}
Now we prove some properties with respect to the estimate of  $\rho_\infty$.
\begin{lem}\label{L71}
(i)$\ird{|x|^{\lambda}\rho_{\infty}}$ is uniformly bounded for $q\in (\frac{2N}{2N+\lambda},1)$.\quad (ii)$\lim\limits_{q\to 1_{-}}C(q)(1-q)=1$.
\end{lem}
\begin{proof}
For the proof of part (i), we just need to consider the case when $q\to 1_{-}$. We first recall the theory of Gamma function and Beta function. Remind that
\begin{equation}
\label{betaf}
B(p, q):= \int_0^\infty \frac {x^{p-1}} {(1+x)^{p+q} } dx,\quad  \Gamma(s):=\int_0^\infty {x^{s-1}}e^{-x} dx
\end{equation}
and 
\[B(p,q)=\frac{\Gamma(p)\Gamma(q)}{\Gamma(p+q)}\]
by mean-value theorem, we have
\[
f(x+h) = f(x) + h \cdot \int_0^1 \nabla f (x+t h) dt= f(x) + \nabla f (x) \cdot h + \int_0^1 \int_0^1 t h \cdot \nabla^2 f(x+t s h) \cdot h ds dt 
\]
and 
\[
\int_0^1 \int_0^1 t h \cdot \nabla^2 f(x+t s h) \cdot h ds dt = \int_0^1 \int_0^t  h \cdot \nabla^2 f(x+ w h) \cdot h d w dt =  \int_0^1 (1-w)  h \cdot \nabla^2 f(x+ w h) \cdot h d w 
\]
notice that
\[
\nabla V_{\lambda}(x) = x |x|^{\lambda-2}, \quad [\nabla^2 V_\lambda(x) ]_{ij}= |x|^{\lambda-2} \delta_{ij} + (\lambda-2) x_i x_j |x |^{\lambda-4},\quad i,j=1,...N
\]
it is easily seen that
\[
|\nabla^2 V_{\lambda}(x) |\le \kappa(N,\lambda) |x|^{\lambda-2} I_N
\]
for some constant $\kappa(N,\lambda)>0$, where $I_N$ is the $N \times N$ identity matrix. Set $y=-h$, we then compute
\begin{equation*}
\begin{aligned}
\int_{\R^N}\frac{1}{\lambda}|x-y|^{\lambda}\rho_{\infty}(y) dy =& \int_{\R^N}\frac{1}{\lambda}|x|^{\lambda}\rho_{\infty}(y) dy -\int_{\R^N} x \cdot y|x|^{\lambda-2} \rho_{\infty}(y) dy 
\\
&+ \int_{\R^N}\rho_{\infty}(y) \int_0^1 (1-w)  y \cdot \nabla^2 V_\lambda(x- w y) \cdot y d w dy
\end{aligned}
\end{equation*}
recall that
\[
\int_{\R^N} \rho_\infty(y) dy = 1, \quad \int_{\R^N} y_i\rho_\infty(y) dy = 0, \quad i=1,2,...,N
\]
so we deduce
\begin{equation}\label{est1}
0 \le \frac{1}{\lambda}|x|^{\lambda}*\rho_{\infty} - \frac 1 \lambda |x|^\lambda \le \kappa\int_{\R^N} \rho_\infty(y)|y|^2 (|x|^{\lambda-2} + |y|^{\lambda-2}) dy\le C_0(N,\lambda) |x|^{\lambda-2} +C_1(N,\lambda)
\end{equation}
with 
\[
C_0 \le \kappa\int_{\R^N} \rho_\infty(y) |y|^2 dy, \quad C_1\le \kappa\int_{\R^N} \rho_\infty(y) |y|^\lambda dy.
\]
in particular,
\begin{equation}
\label{est1}
\frac 1 \lambda |x|^\lambda +C \le  \frac{q}{1-q}\rho_{\infty}^{q-1}=\frac{1}{\lambda}|x|^{\lambda}*\rho_{\infty}+C \le\frac 1 \lambda |x|^\lambda +C +C_0|x|^{\lambda-2}+C_1
\end{equation}
which implies
\begin{equation}
\label{imp1}
\rho_{\infty}<\left(\frac{1-q}{q}\right)^\frac{1}{q-1}\left(\frac{|x|^\lambda}{\lambda}+C\right)^\frac{1}{q-1}
\end{equation}
from $\ird{\rho_{\infty}}=1$, we have 
\[
\left(\frac{1-q}{q}\right)^\frac{1}{q-1}(\lambda C)^\frac{N}{\lambda}C^\frac{1}{q-1}\ird{(1+|x|^\lambda)^\frac{1}{q-1}}>1
\]
by \eqref{betaf} we have
\[
\ird{(1+|x|^\lambda)^\frac{1}{q-1}}
=\frac{S_N}{\lambda}B\left(\frac{N}{\lambda},\frac{1}{1-q}-\frac{N}{\lambda}\right)\]
so
\[
C^{\frac{1}{1-q}-\frac{N}{\lambda}}<S_N\lambda^{\frac{N}{\lambda}-1}\Gamma\left(\frac{N}{\lambda}\right)\left(\frac{1-q}{q}\right)^\frac{1}{q-1}\frac{\Gamma\left(\frac{1}{1-q}-\frac{N}{\lambda}\right)}{\Gamma\left(\frac{1}{1-q}\right)}\sim \frac{S_N}{e}\lambda^{\frac{N}{\lambda}-1}\Gamma\left(\frac{N}{\lambda}\right)\left(\frac{1}{1-q}\right)^{\frac{1}{1-q}-\frac{N}{\lambda}}
\]
here we use the facts that as $q\to 1_{-}$,
\[
\lim_{q\to1_{-}}q^{\frac{1}{1-q}}=\frac{1}{e},\quad \frac{\Gamma\left(\frac{1}{1-q}-\frac{N}{\lambda}\right)}{\Gamma\left(\frac{1}{1-q}\right)} \sim (1-q)^\frac{N}{\lambda}
\]
so we conclude
\begin{equation}
\label{estimp1}
\lim_{q\to 1_{-}}C(1-q)\le 1.
\end{equation}
On the other hand, remind that $\rho_{\infty}$ satisfies $\Delta(\rho_{\infty}^q)+\nabla(\rho_{\infty}\nabla V_{\lambda}*\rho_{\infty})=0$, so
\[
\ird{|x|^2\Delta(\rho_{\infty}^q)}+\ird{|x|^2\nabla(\rho_{\infty}\nabla V_{\lambda}*\rho_{\infty})}=0
\]
we compute
\begin{equation*}
\begin{aligned}
\ird{|x|^2\nabla(\rho_{\infty}\nabla V_{\lambda}*\rho_{\infty})} =& -\ird{ \ire{ 2x \cdot (x-y) |x-y|^{\lambda-2} \rho_{\infty}(x)\rho_{\infty}(y) }} \\
=& -\ird{ \ire{ 2y \cdot (y-x) |x-y|^{\lambda-2}\rho_{\infty}(x)\rho_{\infty}(y) }} \\
=& -\ird{ \ire{  |x-y|^{\lambda}\rho_{\infty}(x)\rho_{\infty}(y) }} \\
\end{aligned}
\end{equation*}
where we change $x$ and $y$ in the second equality. So after integrating by parts, 
\begin{equation}
\label{eqd1}
2N\ird{\rho_{\infty}^q}=\ire{\ird{|x-y|^{\lambda}\rho_{\infty}(x)\rho_{\infty}(y)}}
\end{equation}
multiply \eqref{remindsta} by $\rho_{\infty}$ and integrate, we have
\begin{equation}
\label{eqd2}
\frac{q}{1-q}\ird{\rho_{\infty}^q}=\frac{1}{\lambda}\ire{\ird{|x-y|^{\lambda}\rho_{\infty}(x)\rho_{\infty}(y)}}+C
\end{equation}
from \eqref{eqd1}.\eqref{eqd2}, 
\[
\ire{\ird{|x-y|^{\lambda}\rho_{\infty}(x)\rho_{\infty}(y)}}=\frac{C(1-q)2N\lambda}{q\lambda-2N(1-q)}\lesssim 2N
\]
recall $\rho_\infty$ is  radially symmetric, non-increasing, by rearrangement inequality
\[
\ire{\ird{|x-y|^{\lambda}\rho_{\infty}(x)\rho_{\infty}(y)}}\ge \ird{|x|^{\lambda}\rho_{\infty}}
\]
so $\ird{|x|^{\lambda} \rho_{\infty} }$ is uniformly bounded, and the part (i) of the lemma is proved. By H\"older inequality we deduce that
\[
C_1,C_0\le \zeta(N,\lambda)
\]
that does not depend on $q$, so by \eqref{est1}
\[
\rho_{\infty}>\left(\frac{1-q}{q}\right)^\frac{1}{q-1}\left(\frac{|x|^\lambda}{\lambda}+C_0|x|^{\lambda-2}+C_1+C\right)^\frac{1}{q-1}
\]
notice that there exists $C_2(N,\lambda)>0$ that does not depend on $q$, such that for any $x\in\mathbb R^N$. \[
C_0|x|^{\lambda-2}\le \frac{|x|^\lambda}{\lambda}+C_2
\]
define $C_3(N,\lambda):=C+C_1(N,\lambda)+C_2(N,\lambda)$. Then
\[
\rho_{\infty}>\left(\frac{1-q}{q}\right)^\frac{1}{q-1}\left(\frac{2|x|^\lambda}{\lambda}+C_3\right)^\frac{1}{q-1}
\]
set $x=\left(\frac{C_3\lambda}{2}\right)^\frac{1}{\lambda}y$, by a similar calculation and \eqref{betaf}
\[
\left(\frac{1-q}{q}\right)^\frac{1}{q-1}\left(\frac{C_3\lambda}{2}\right)^\frac{N}{\lambda}C_3^\frac{1}{q-1}\frac{S_N}{\lambda}B\left(\frac{N}{\lambda},\frac{1}{1-q}-\frac{N}{\lambda}\right)<1
\]
which implies
\[
C_3^{\frac{1}{1-q}-\frac{N}{\lambda}}>\left(\frac{1-q}{q}\right)^\frac{1}{q-1}\left(\frac{\lambda}{2}\right)^\frac{N}{\lambda}\frac{S_N}{\lambda}\frac{\Gamma\left(\frac{N}{\lambda}\right)\Gamma\left(\frac{1}{1-q}-\frac{N}{\lambda}\right)}{\Gamma\left(\frac{1}{1-q}\right)}\sim (1-q)^{\frac{1}{q-1}+\frac{N}{\lambda}}\left(\frac{\lambda}{2}\right)^\frac{N}{\lambda}\frac{S_N}{\lambda e}\Gamma\left(\frac{N}{\lambda}\right)
\]
as $q\to 1_{-}$. So we have
\[
\lim_{q\to 1_{-}}C_3(1-q)\ge 1
\]
since $C_0,C_1,C_2$ do not depend on $q$, 
\begin{equation}
\label{estimp2}
\lim_{q\to 1_{-}} C(1-q)\ge 1.
\end{equation}
and the part (ii) is directly deduced from \eqref{estimp1},\eqref{estimp2}.
\end{proof}

\begin{lem}\label{lem7.2}
There exists a constant $\Lambda(N,\lambda)>0$, such that for all $q \in (\frac{2N}{2N+\lambda}, 1)$
\begin{equation}
\label{keyes1}
\ire{\ird{|x-y|^{2\lambda}  \rho_{\infty}^{2-q} (x)   \rho_{\infty}^{2-q} (y) }}  \le \Lambda,\quad\ire{\ird{|x-y|^{2\lambda-2}  \rho_{\infty} (x)   \rho_{\infty}^{2-q} (y) }}  \le \Lambda.
\end{equation}
\end{lem}
\begin{proof}
We only need to prove the case that $q\to 1_{-}$. Notice that
\[
|x-y|^p\le 2^{p-1}(|x|^p+|y|^p)\quad\mbox{for}\quad p\ge 1
\]
so we need to compute 
\[
\ird{|x|^{2\lambda}\rho^{2-q}_{\infty}}\ird{\rho^{2-q}_{\infty}},\quad \ird{|x|^{2\lambda-2}\rho_{\infty}}\ird{\rho^{2-q}_{\infty}}+\ird{|x|^{2\lambda-2}\rho^{2-q}_{\infty}}
\]
after interpolation, we only need to estimate the integrals
\[
\ird{|x|^{2\lambda-2}\rho_{\infty}},\quad\ird{\rho^{2-q}_{\infty}},\quad\ird{|x|^{2\lambda}\rho^{2-q}_{\infty}},
\]
first, from \eqref{ss4}, for any $x\in \mathbb R^N$,
\[
\frac{q}{1-q}\rho_{\infty}^{q-1}\ge C
\]
so as $q\to 1_{-}$,
\[
\rho_{\infty}^{1-q}\le \frac{q}{C(1-q)}\to 1
\]
which is
\[
\ird{\rho_{\infty}^{2-q}}\lesssim \ird{\rho_{\infty}}=1,\quad \ird{|x|^{2\lambda}\rho_{\infty}^{2-q}}\lesssim \ird{|x|^{2\lambda}\rho_{\infty}}
\]
so after interpolation, we only need to estimate $ \ird{|x|^{2\lambda}\rho_{\infty}}$. Remind that

\[
\left(\frac{1-q}{q}\right)^\frac{1}{q-1}\left(\frac{2|x|^\lambda}{\lambda}+C_3\right)^\frac{1}{q-1}<\rho_{\infty}<\left(\frac{1-q}{q}\right)^\frac{1}{q-1}\left(\frac{|x|^\lambda}{\lambda}+C\right)^\frac{1}{q-1}
\]
similarly as the computation above, for $q$ close enough to 1, by using $C\sim\frac{1}{1-q}$ and $C_1 , C_2 \le \zeta(N,\lambda)$, we have
\[
\frac{\ird{\left(\frac{|x|^\lambda}{\lambda}+C\right)^\frac{1}{q-1}}}{\ird{\left(\frac{2|x|^\lambda}{\lambda}+C_3\right)^\frac{1}{q-1}}}
=2^\frac{N}{\lambda}\cdot\frac{\ird{\left(\frac{|x|^\lambda}{\lambda}+C\right)^\frac{1}{q-1}}}{\ird{\left(\frac{|x|^\lambda}{\lambda}+C_3\right)^\frac{1}{q-1}}}
=2^\frac{N}{\lambda}\cdot\left(\frac{C+C_1+C_2}{C}\right)^{\frac{1}{1-q}-\frac{N}{\lambda}}\sim 2^\frac{N}{\lambda}e^{C_1+C_2}
\]
recall that $\ird{\rho_{\infty}}=1$, we have
\[
\left(\frac{1-q}{q}\right)^\frac{1}{q-1}\ird{\left(\frac{|x|^\lambda}{\lambda}+C\right)^\frac{1}{q-1}}\lesssim 2^\frac{N}{\lambda}e^{C_1+C_2}
\]
so from \eqref{imp1}, we obtain that as $q\to 1$,
\begin{equation*}
\begin{aligned}
\ird{|x|^{2\lambda}\rho_{\infty}}&<\left(\frac{1-q}{q}\right)^\frac{1}{q-1}\ird{|x|^{2\lambda}\left(\frac{|x|^\lambda}{\lambda}+C\right)^\frac{1}{q-1}}\\
&=(C\lambda)^2\cdot\frac{\Gamma(2+\frac{N}{\lambda})\Gamma(\frac{1}{1-q}-2-\frac{N}{\lambda})}{\Gamma(\frac{N}{\lambda})\Gamma(\frac{1}{1-q}-\frac{N}{\lambda})}\cdot\left(\frac{1-q}{q}\right)^\frac{1}{q-1}\ird{\left(\frac{|x|^\lambda}{\lambda}+C\right)^\frac{1}{q-1}}\\
&\sim (C(1-q))^2\frac{\lambda^2\Gamma(2+\frac{N}{\lambda})}{\Gamma(\frac{N}{\lambda})}\cdot\left(\frac{1-q}{q}\right)^\frac{1}{q-1}\ird{\left(\frac{|x|^\lambda}{\lambda}+C\right)^\frac{1}{q-1}}\\
&\lesssim \frac{\lambda^2\Gamma(2+\frac{N}{\lambda})}{\Gamma(\frac{N}{\lambda})}2^\frac{N}{\lambda}e^{C_1+C_2}.
\end{aligned}
\end{equation*}
so the proof is thus finished.
\end{proof}

\begin{rmk}
For the special case $\lambda =4$, we have that after translation,
\begin{equation}
\label{ss4}
\frac{q}{1-q}\rho_{\infty}^{q-1}=\frac{1}{4}|x|^4+\frac{3a}{2}|x|^2+C
\end{equation}
here $a$ satisfies $\ird{(|x|^2-a)\rho_\infty}=0$, which is
\[
\int_{0}^{\infty}{(x^{N+1}-ax^{N-1})\left(\frac{1}{4}x^4+\frac{3a}{2}x^2+C\right)^\frac{1}{q-1}dx}=0
\]
so we can prove the case $\lambda=4$ directly.
\end{rmk}

\subsection{General Hardy-Poincar\'e inequality}
The main goal of this subsection is to show that
\begin{lem}\label{lem7.4}Recall that for $\lambda>2$
\[
\mathcal M(x) =\frac 1 {1+|x|^{\lambda-2}}
\]
for any $q\in(\frac{2N}{2N+\lambda},1)$, there exists a constant $\mathcal C_{q, N,\lambda}>0$, such that for all $h$ that satisfies
\[
\ird{\rho_\infty^{2-q} h} =0,
\]
we have
\[
\ird{\rho_\infty|\nabla h|^2 \mathcal{M} } \ge \mathcal C_{q, N,\lambda} \ird{\rho_\infty^{2-q}h^2}.
\]
Moreover, there exists a constant $\Omega(N,\lambda)>0$, such that for any $q\in(\frac{2N}{2N+\lambda},1)$,
\begin{equation}
\label{keyes2}
\mathcal C_{q, N,\lambda}(1-q) \ge \Omega(N,\lambda).
\end{equation}
\end{lem}
This lemma is inspired from \cite{blanchet2009asymptotics} Appendix A, where the authors proved the similar result for the case $\lambda=2$.
\begin{thm}(\cite{blanchet2009asymptotics} Appendix A, Hardy-Poincar\'e inequality) Let
\[
V_D(x) = \left(D+\frac {1-q} {2q} |x|^2\right)^{-\frac 1 {1-q} } ,\quad  d\mu = V_D^{2-q}dx, \quad d \nu = V_D dx
\]
Let $N \ge 1$ and $D>0$. If $q \in (\max\left\{ \frac {N-4} {N-2} , 0\right\},1)$, then there exists a constant $\mathcal A_{q, N }$ which does not depend on $D$ such that for all smooth function $g$ we have
\[
\mathcal A_{q, N}\int_{\R^N} |g-\bar {g} |^2 d \mu \le \int_{\R^N} |\nabla g |^2 d \nu, \quad \bar{g} = \int_{\R^N} g d \mu
\]
where 
\[
\mathcal A_{q, N } = \frac {(Nq-2q -N+4 )^2} {8q(1-q)}
\]
is optimal.
\end{thm}
\noindent Let $q$ tends to 1 we have
\[
\lim_{q\to 1_{-}}(1-q)\mathcal A_{q, N}=\frac 12.
\]
Before proving the lemma~\ref{lem7.4} we still need
\begin{lem}\label{lem7.7} (\cite{barthe2006modified} Theorem 1, 2) Let $\mu, \nu$ be two Borel measures on $\R^+$, then the best constant $A$ such that every smooth function $f$ verifies 
\[
\int_0^{+\infty} |f-f(0)|^2 d\mu \le A \int_0^{+\infty}|f'|^2 d\nu
\]
is finite if and only if
\[
B:=\sup_{x>0} \mu([x, +\infty ))\int_0^x \frac 1 {n(s)} ds
\]
is finite, where $n$ is the density of the absolute continuous part of $\nu$, moreover when it is finite we have 
\[
B \le A \le 4B 
\]
\end{lem}
\begin{proof}({\bf Proof of Lemma \ref{lem7.4}})
The existence of $\mathcal C(q,N,\lambda)$ can be deduced as Lemma~\ref{lemhp}, so we mainly focus on the behavior of $\mathcal C(q,N,\lambda)$ as $q\to 1_{-}$. The proof is the same as the Hardy-Poincar\'e inequality, we omit some details and focus on the asymptotic behavior. We only prove for the case $N=1$, and for the proof of general $N$, we should prove  the results in radical functions plus the Poincar\'e inequality on the unit sphere
\[
\int_{\mathbb{S}^{N-1}} \left|   u-\hat{u} \right|^2 d\theta \le \frac 1 {N-1}\int_{\mathbb{S}^{N-1}} \left|   \nabla_\theta u \right|^2 d\theta
\]
with 
\[
\hat{u} = \int_{\mathbb{S}^{N-1}}  u d\theta
\]
the full proof for general $N$ can also be found in \cite{blanchet2009asymptotics} Appendix A. Define
\[
\mathcal R(x):=\frac{q}{1-q}\rho_{\infty}^{q-1}-\frac{1}{\lambda}|x|^{\lambda}-C
\]
we only need to prove that
\[
\mathcal C_{q, 1,\lambda} \frac q{1-q} \int_0^\infty |h(r) - \bar{h}|^2 \left(\frac{1}{\lambda}r^{\lambda}+\mathcal R(r)+C\right)^{-\frac {2-q} {1-q}} dr\le \int_0^\infty |h'(r)|^2  \left(\frac{1}{\lambda}r^{\lambda}+\mathcal R(r)+C\right)^{-\frac {1} {1-q} } \frac 1 {1+r^{\lambda-2} }dr
\]
set $r^{\lambda}=Cw^{\lambda}$, the former inequality becomes
\begin{equation*}
\begin{aligned}
&\mathcal C_{q, 1,\lambda} \frac q{1-q} \int_0^\infty |h(w) - \bar{h}|^2\left(\frac{1}{\lambda}w^{\lambda}+\frac{\mathcal R(C^{\frac{1}{\lambda}}w)}{C}+1\right)^{-\frac {2-q} {1-q}}  dw
\\
&\le \int_0^\infty |h'(w)|^2 \left(\frac{1}{\lambda}w^{\lambda}+\frac{\mathcal R(C^{\frac{1}{\lambda}}w)}{C}+1\right)^{-\frac {1} {1-q} }  \left(\frac {1} {C^{\frac {2-\lambda} \lambda}  +w^{\lambda-2  }  }\right) d w
\end{aligned}
\end{equation*}
when $q$ is close enough to 1,  we have $C \ge 1$, and since $\lambda >2$
\[
\frac {1} {1 +w^{ \lambda-2  }}  \le \frac {1} {C^{\frac {2-\lambda} \lambda}  +w^{\lambda-2  } }
\]
so it is enough to prove that
\begin{equation*}
\begin{aligned}
&\mathcal C_{q, 1,\lambda} \frac q{1-q} \int_0^\infty |h(w) - \bar{h}|^2\left(\frac{1}{\lambda}w^{\lambda}+\frac{\mathcal R(C^{\frac{1}{\lambda}}w)}{C}+1\right)^{-\frac {2-q} {1-q}} d w
\\
&\le \int_0^\infty |h'(w)|^2 \left(\frac{1}{\lambda}w^{\lambda}+\frac{\mathcal R(C^{\frac{1}{\lambda}}w)}{C}+1\right)^{-\frac {1} {1-q} }\mathcal M(w)d w
\end{aligned}
\end{equation*}
remind
\[
|\mathcal{R}(w)| \le \max\{C_0(\lambda),C_1(\lambda)\}\cdot(w^{\lambda-2} +1)
\]
where $C_0,C_1$ do not depend on $q$, choose 
\[
a(\lambda):=\max\{2\lambda, 1+\max\{C_0(\lambda),C_1(\lambda)\}\}
\]
then for all $w\ge 0$,
\[
\frac 1 {a(\lambda) }(1+w^2) (1+w^{\lambda-2}) \le \frac{1}{\lambda}w^{\lambda}+\frac{\mathcal R(C^{\frac{1}{\lambda}}w)}{C}+1 \le a(\lambda)(1+w^2) (1+w^{\lambda-2})
\]
we only need to prove that
\begin{equation*}
\begin{aligned}
&\frac{a(\lambda)q\mathcal C_{q, 1,\lambda}}{1-q} \int_0^\infty |h(w) - \bar{h}|^2 \left(\frac{1}{\lambda}w^{\lambda}+\frac{\mathcal R(C^{\frac{1}{\lambda}}w)}{C}+1\right)^{-\frac {1} {1-q}} (1+w^2)^{-1} d w
\\
&\le \int_0^\infty |h'(w)|^2  \left(\frac{1}{\lambda}w^{\lambda}+\frac{\mathcal R(C^{\frac{1}{\lambda}}w)}{C}+ 1 \right)^{-\frac {2-q} {1-q} }\mathcal M(w)d w
\end{aligned}
\end{equation*}
by Lemma \ref{lem7.7}, we have that
\[
\mathcal C_{q, 1,\lambda} \frac q{1-q}  \ge \frac {1} {4a(\lambda)K}
\]
with 
\[
K =\max_{r>0} \int_0^r  \left(\frac{1}{\lambda}w^{\lambda}+\frac{\mathcal R(C^{\frac{1}{\lambda}}w)}{C} + 1\right)^{\frac {2-q} {1-q} } (1+w^2)^{-1}  d w \int_r^\infty \left(\frac{1}{\lambda}w^{\lambda}+\frac{\mathcal R(C^{\frac{1}{\lambda}}w)}{C} + 1\right)^{-\frac {2-q} {1-q} } d w
\]
it's easily seen that 
\[
K\sim\max_{r>0} \int_0^r w^{\frac {\lambda(2-q)} {1-q}-2} d w \int_r^\infty w^{-\frac {\lambda(2-q)} {1-q} }d w =O((1-q)^2)
\]
so
\[
\lim_{q\to 1_{-}}\mathcal C_{q,1,\lambda}(1-q)\ge \frac{\lambda^2}{4a(\lambda)}
\]
which finishes the proof of the lemma.
\end{proof}

\subsection{The main result (proof of  Theorem\ref{thm1.4})}
We first prove that
\begin{lem}\label{lem8.1}
For any $\lambda>2$, there exists $q_{N,\lambda}\in (\frac{2N}{2N+\lambda}, 1)$, such that for any $q\in (q_{N,\lambda}, 1)$ and $v(t)$ as the solution of \eqref{changev1}, there exists $\gamma>0$, such that
\[
\mathbb F[v] \le \gamma \mathbb  I[v].
\]
\end{lem}
\begin{proof}
Recall
\[
\rho = \rho_\infty v = \rho_\infty (1+g) = \rho_\infty + j
\]
by Cauchy-Schwarz inequality, 
\[
\ire{|x-y|^\lambda j(y)} \le \(\ire{ |x-y|^{2\lambda} \rho_{\infty}^{2-q} (y)} \)^{\frac 1 2}  \(\ire{  \rho_{\infty}^{q-2}  j^2 }  \)^{\frac 1 2}
\]
still by Cauchy-Schwarz inequality and Lemma \ref{lem7.2} we have
\begin{equation*}
\begin{aligned}
\left| \ire{\ird{|x-y|^\lambda j(x)j(y)}} \right|  &\le    \(\ird{  \rho_{\infty}^{q-2}  j^2 }  \)^{\frac 1 2} \(\ird{  \rho_{\infty}^{2-q} (x) \left|\ire{|x-y|^\lambda  j(y)} \right|^2} \)^{\frac 1 2}\\
&\le \(\ire{\ird{|x-y|^{2\lambda}  \rho_{\infty}^{2-q} (x)   \rho_{\infty}^{2-q} (y) }}  \)^{\frac 1 2}\(\ire{  \rho_{\infty}^{q-2}  j^2 }  \)
\end{aligned}
\end{equation*}
so together with Lemma \ref{lem5.3}, 
\begin{equation}
\label{keyes3}
\mathbb F[v] \le \left(\frac{1}{2}\mathcal W_0^{q-2}+\Lambda\right)\ird{\rho_\infty^{q-2}j^2}.
\end{equation}
Similarly still by Cauchy-Schwarz inequality
\[
|\nabla V_{\lambda}*j| (x)=  \left|\ire{ |x-y|^{\lambda-2} (x-y) j(y)} \right| \le \(\ire{ |x-y|^{2\lambda-2} \rho_{\infty}^{2-q} (y)} \)^{\frac 1 2}  \(\ire{  \rho_{\infty}^{q-2}  j^2 }  \)^{\frac 1 2}
\]
so we have
\[
\ird{\rho_\infty\left|\nabla V_{\lambda}*j\right|^2} \le\(\ire{\ird{|x-y|^{2\lambda-2}  \rho_{\infty} (x)   \rho_{\infty}^{2-q} (y) }}  \)\(\ire{  \rho_{\infty}^{q-2}  j^2 }  \)
\]
so by Lemma \ref{lem7.2},
\begin{equation}
\label{keyes4}
\ird{\rho_\infty\left|\nabla V_{\lambda}*j\right|^2} \le \Lambda\ird{\rho_\infty^{q-2}j^2}
\end{equation}
since
\[
|a+b|^2 \le 2|a|^2 +2|b|^2
\]
choose $\epsilon\in(0,\frac{1}{8}q^2)$ small enough, so that $\epsilon$ satisfies \eqref{keyes2} of Lemma \ref{lem5.4}. Together with \eqref{keyes4}, we have
\begin{equation*}
\begin{aligned}
\mathbb{I}[v] =&\ird{\rho_\infty v \left|\frac{q}{1-q}\nabla(\rho_\infty ^{q-1}(v^{q-1} -1))- \nabla V_\lambda*j\right|^2}
\\
\ge&\ird{\rho_\infty v \left|\frac{q}{1-q}\nabla(\rho_\infty ^{q-1}(v^{q-1} -1))- \nabla V_\lambda*j\right|^2 \mathcal M}
\\
\ge&\frac 1 2\ird{\rho_\infty\left|q\nabla(\rho_\infty^{q-2}j)-\nabla V_{\lambda}*j\right|^2 \mathcal M}- \frac 12\ird{\rho_\infty \left|\nabla \left(\rho_\infty ^{q-1}\frac {(v^{q-1} -1)} {q-1}  - \nabla(\rho_\infty^{q-1}(v-1))   \right)\right|^2 \mathcal M} 
\\
\ge&\frac 1 2\ird{\rho_\infty\left|q\nabla(\rho_\infty^{q-2}j)-\nabla V_{\lambda}*j\right|^2 \mathcal M}-\epsilon \Psi_1(g) - \epsilon \Psi_2(g) 
\\
\ge& \frac  1 4 q^2\ird{\rho_\infty|\nabla(\rho_\infty^{q-2}j)|^2 \mathcal M}- \frac 1 2 \ird{\rho_\infty\left|\nabla V_{\lambda}*j\right|^2} -\epsilon \Psi_1(g) - \epsilon \Psi_2(g) 
\\
\ge& \frac  1 8 q^2\ird{\rho_\infty|\nabla(\rho_\infty^{q-2}j)|^2 \mathcal M} - \Lambda\ird{\rho_\infty^{q-2}j^2}
\end{aligned}
\end{equation*}
since 
\[
\ird { (\rho_\infty^{q-2}j) \rho_\infty^{2-q} } = \ird{j} = \ird{(\rho-\rho_\infty)} = 0
\]
and by Lemma \ref{lem7.4},
\[
\ird{\rho_\infty|\nabla(\rho_\infty^{q-2}j)|^2 \mathcal M} \ge \mathcal C_{q, N,\lambda} \ird{\rho_\infty^{q-2}j^2}
\]
where
\[
\mathcal C_{q, N,\lambda}\ge\frac{\Omega(N,\lambda)}{1-q}
\]
set 
\[
q(N,\lambda):=\mbox{max}\left\{ 1-\frac{\Omega(N,\lambda)}{16\Lambda(N,\lambda)},\frac{2N}{2N+\lambda} \right\}
\] 
recall that $\Omega, \Lambda$ are defined in \eqref{keyes1},\eqref{keyes2}. Then for any $q \in (q(N, \lambda),1)$, $\mathcal C_{q,N,\lambda}\ge 16 \Lambda$. Finally by using \eqref{keyes3}, we have
\[
\mathbb I [v]  \ge \frac 1 8 \Psi_2[g] -\Lambda\Psi_1[g] \ge \frac {\mathcal C_{q, N,\lambda}} {16} \Psi_1 [g] \ge \frac {\mathcal C_{q, N,\lambda}} {16\left(\frac{1}{2}\mathcal W_0^{q-2}+\Lambda\right)} \mathbb F [v] 
\]
the lemma is thus proved.
\end{proof}
The first part of Theorem \ref{thm1.4} can be directly deduced by using Gr\"onwall inequality. For $\lambda \in[2, 4]$, we conclude the proof of the second part by using
\begin{thm}\label{Thmsp}(\cite{lopes2017uniquenes}, Theorem 2.4) For any smooth $f$  satisfies 
\[
\ird {f(x)} = 0, \quad  \ird {x_if(x)} = 0, i=1,2,...,N
\]
then for $\lambda \in[2, 4]$, we have
\[
\ird {\ire{|x-y|^\lambda f(x) f(y)  } } \ge 0.
\]
\end{thm}
While Theorem~\ref{Thmsp} is not true for $\lambda >4$.




\newpage



\begin{thebibliography}{99} 

\bibitem{barthe2006modified}
Barthe, Franck and Roberto, Cyril.  
{\em Modified Logarithmic Sobolev Inequalities on R.} Potential Analysis. 29. 167-193, 2006

\bibitem{blanchet2007hardy}
Blanchet, Adrien and Bonforte, Matteo and Dolbeault, Jean and Grillo, Gabriele and V{\'a}zquez, Juan Luis. {\em Hardy--Poincar{\'e} inequalities and applications to nonlinear diffusions}, Comptes Rendus Math{\'e}matique (2007), no.~7, 431--436.
\bibitem{blanchet2009asymptotics}
Blanchet, Adrien and Bonforte, Matteo and Dolbeault, Jean and Grillo, Gabriele and V{\'a}zquez, Juan Luis. {\em Asymptotics of the fast diffusion equation via entropy estimates},  Archive for Rational Mechanics and Analysis. {\bf 191} (2009), no.~2, 347--385.
\bibitem{blanchet2009critical}
Blanchet A., Carrillo J.A., and Laurencot P.. {\em Critical mass for a Patlak-Keller-Segel model with degenerate diffusion in higher dimensions.} Calc. Var. Partial Differential Equations, 35(2):133-168, 2009.




\bibitem{bonforte2006global}
Bonforte, Matteo and Vazquez, Juan Luis. {\em Global positivity estimates and Harnack inequalities for the fast diffusion equation}, Journal of Functional Analysis(2006),
no.~2, 399--428.

\bibitem{calvez2017equilibria}
Calvez, Vincent and Carrillo, Jose Antonio and Hoffmann, Franca. {\em Equilibria of homogeneous functionals in the fair-competition regime}, Nonlinear Analysis (2017), vol~159, 85-128.
\bibitem{calvez2019uniqueness}
Calvez, Vincent and Carrillo, Jose Antonio and Hoffmann, Franca. {\em Uniqueness of stationary states for singular Keller-Segel type models}, arXiv preprint arXiv:1905.07788, 2019.

\bibitem{carrapatoso2017landau}
Carrapatoso, K., Mischler, S. {\em Landau Equation for Very Soft and Coulomb Potentials Near Maxwellians.} Ann. PDE 3, 1 (2017)

\bibitem{carrillo2020fast}
Carrillo, J, A, Delgadino, M. G., Frank,R. L. and Lewin, M. {\em Fast Diffusion leads to partial mass concentration in Keller-Segel type stationary solutions.} arXiv:2012.08586


\bibitem{carrillo2015ground}
Carrillo J. A, Castorina D. and Volzone B.
 {\em Ground states for diffusion dominated free energies with logarithmic interaction.} SIAM J. Math. Anal., 47(1):1-25, 2015.

\bibitem{dolbeault2018reverse}
Carrillo J. A, Delgadino MG, Dolbeault, Jean, Frank, Rupert and Hoffmann, Franca. {\em Reverse Hardy-littlewood-sobolev inequalities}, Journal de Math\'ematiques Pures \'et Appliqu\'ees, 132:133-165,2019.

\bibitem{carrillo2000asymptotic}
Carrillo J. A, and Toscani, G. {\em Asymptotic $L^1$-decay of solutions of the porous medium equation to self-similarity}, Indiana University Mathematics Journal, 113-142,2000.

\bibitem{catrina2001caffarelli}
Catrina, Florin and Wang, Zhi-Qiang. {\em On the Caffarelli-Kohn-Nirenberg inequalities: Sharp constants, existence (and nonexistence), and symmetry of extremal functions}, Communications on Pure and Applied Mathematics: A Journal Issued by the Courant Institute of Mathematical Sciences (2001), no.~2, 229-258.

\bibitem{dou2015reverse}
Dou, Jingbo and Zhu, Meijun. {\em Reversed Hardy-Littewood-Sobolev inequality}, Int. Math. Res. Not.(2015), no.~19, 9696-9726.

\bibitem{herau2020regularization}
H\'erau, F., Tonon, D. and Tristani, I. {\em Regularization Estimates and Cauchy Theory for Inhomogeneous Boltzmann Equation for Hard Potentials Without Cut-Off.} Commun. Math. Phys. 377, 697-771 (2020)

\bibitem{herrero1985cauchy}
Herrero, Miguel A and Pierre, Michel. {\em The Cauchy problem for $u_t=\Delta(u^m)$ when $0< m< 1$}, Transactions of the American Mathematical Society(1985),no.~1,145--158.



\bibitem{hoffmann2017keller}
Hoffmann, Franca. {\em Keller-Segel-Type Models and Kinetic Equations for Interacting Particles: Long-Time Asymptotic Analysis}, University of Cambridge, 2017.

\bibitem{kim2012the}
Kim I. and Yao Y.
{\em The Patlak-Keller-Segel model and its variations: properties of solutions
via maximum principle.} SIAM Journal on Mathematical Analysis, 44(2):568-602, 2012.

\bibitem{lopes2017uniquenes}
Lopes, Orlando. {\em Uniqueness and radial symmetry of minimizers for a nonlocal variational problem}, Commun. Pure Appl. Anal. 18 (5), 2265-2282.
\bibitem{nirenberg1959on}
Nirenberg, L. {\em On elliptic partial differential equations.} Ann. Scuola Norm. Sup. Pisa
13(3), 115-162, 1959.
\bibitem{ngo2017sharp}
Ng{\^o}, Qu\^oc Anh and Van Hoang Nguyen. {\em Sharp reversed Hardy--Littlewood--Sobolev inequality on $\mathbb R^n$}, Israel Journal of Mathematics (2017), no~1, 189--223.

\bibitem{vazquez2006smoothing}
V{\'a}zquez, Juan Luis. {\em Smoothing and decay estimates for nonlinear diffusion equations: equations of porous medium type}, Oxford University Press, 2006.

\end{thebibliography}
\end{document}